\documentclass[12pt]{amsart}

\usepackage{amssymb,amsmath,amsthm}

\newtheorem{theorem}{Theorem}[section]
\newtheorem{lemma}[theorem]{Lemma}
\newtheorem{prop}[theorem]{Proposition}
\newtheorem{cor}[theorem]{Corollary}

\def \R{\mathbb{R}}

\def \Rn{\mathbb{R}^n}
\def \Rno{\mathbb{R}^n_0}
\def \RnI{\mathbb{R}^{n+1}}
\def \RnIp{\mathbb{R}^{n+1}_{+}}

\def \a{\alpha}
\def \b{\beta}
\def \d{\delta}
\def \g{\gamma}

\def \e{\varepsilon}
\def \ph{\varphi}

\def \s{\sigma}

\def \L{\mathcal{L}}

\def \z{\zeta}

\def \Om{\Omega}

\def \Lap{\triangle}
\def \grad{\nabla}

\def \S{\mathcal{S}}
\def \Czinf{C^{\infty}_0}

\def \Omt{\tilde{\Omega}}

\def \half{\frac{1}{2}}

\def \Lph{\mathcal{L}_{\ph}}
\def \Lphaq{\mathcal{L}_{A,q,\ph}}
\def \Lphaqe{\mathcal{L}_{A,q,\ph_c}}
\def \Laq{\mathcal{L}_{A,q}}
\def \Lphe{\mathcal{L}_{\ph_c}}

\def \Lphet{\tilde{\mathcal{L}}_{\ph_c}}

\numberwithin{equation}{section}

\begin{document}


\title[Partial Data Magnetic Neumann-Dirichlet]{Partial Data for the Neumann-Dirichlet Magnetic Schr\"{o}dinger Inverse Problem}

\author[Chung]{Francis J. Chung}
\address{Department of Mathematics, University of Michigan, Ann Arbor, MI}



\begin{abstract}
We show that an electric potential and magnetic field can be uniquely determined by partial boundary measurements of the Neumann-to-Dirichlet map of the associated magnetic Schr\"{o}dinger operator.  This improves upon the results in ~\cite{Ch2} by including the determination of a magnetic field.  The main technical advance is an improvement on the Carleman estimate in ~\cite{Ch2}.  This allows the construction of complex geometrical optics solutions with greater regularity, which are needed to deal with the first order term in the operator.  This improved regularity of CGO solutions may have applications in the study of inverse problems in systems of equations with partial boundary data.    

\end{abstract}

\maketitle

\section{Introduction}

Let $\Om$ be a smooth simply-connected bounded domain in $\R^{n+1}$, where $n+1 \geq 3$.  Let $A$ be a $C^2$ vector field on $\Om$, and $q$ be an $L^{\infty}$ function on $\Om$.  Then define the magnetic Schr\"{o}dinger operator $\Laq$ with magnetic potential $A$ and electric potential $q$ by
\begin{equation}\label{Laq}
\Laq = (D+A)^2 + q
\end{equation}
where $D = -i \grad$.  Let $\nu(p)$ be the outward unit normal at each point $p \in \partial \Om$, and assume $A$ and $q$ are such that the Neumann boundary value problem
\begin{equation*}
\begin{split}
\Laq u &= 0 \mbox{ in } \Om \\
\nu \cdot (\grad + iA) u|_{\partial \Om} &= g \\
\end{split}
\end{equation*}
has unique solutions $u \in H^1(\Om)$ for each $g \in H^{-\half}(\partial \Om)$.  Then $A$ and $q$ define a Neumann-to-Dirichlet map $N_{A,q}: H^{-\half}(\Om) \rightarrow H^{\half}(\Om)$ by 
\[
N_{A,q} g = u|_{\partial \Om}.
\]
The basic inverse problem associated to this map is whether knowledge of $N_{A,q}$ suffices to determine the electric potential $q$ and the magnetic field $dA$.  Here $dA$ makes sense by identifying $A$ with the $1$-form $A_{i}dx^i$.  We will refer to this as the Neumann-to-Dirichlet magnetic Schr\"{o}dinger inverse problem.  

A related problem is the Dirichlet-to-Neumann magnetic Schr\"{o}dinger inverse problem.  Here instead one considers the Dirichlet boundary value problem for $\Laq$, defines a Dirichlet-to-Neumann map $\Lambda_{A,q}$, and asks whether $\Lambda_{A,q}$ determines $q$ and $dA$.  If $A \equiv 0$, so $\Laq = -\Lap + q$, then the Dirichlet-to-Neumann magnetic Schr\"{o}dinger inverse problem is equivalent to Calder\'{o}n's problem, posed in ~\cite{C}, for $C^2$ conductivities.  

For Calder\'{o}n's problem, the fact that the Dirichlet-to-Neumann map determines $q$ was shown by Sylvester and Uhlmann in ~\cite{SU}.  This result was extended to the Dirichlet-to-Neumann magnetic Schr\"{o}dinger inverse problem by Nakamura, Sun, and Uhlmann in ~\cite{NSU}.  This work suffices to solve the Neumann-to-Dirichlet magnetic Schr\"{o}dinger inverse problem as well, since the problems can be shown to be equivalent in the case that $\Lambda_{A,q}$ and $N_{A,q}$ are fully known.  Much more work has been done on this problem since then, expanding this result to the case of less smooth potentials and unbounded domains; see for example ~\cite{Sa_thesis}, ~\cite{KrUh} and ~\cite{KLU}.

A natural follow-up question is to ask whether knowledge of $N_{A,q}$ or $\Lambda_{A,q}$ on a subset of the boundary suffices to recover $q$ and $dA$.  In principle this could mean either of the following:  partial output, where $N_{A,q}g|_{\Gamma}$ is known for some $\Gamma \subset \partial \Om$; or partial input, where $N_{A,q}g$ is known only for $g$ supported on $\Gamma \subset \partial \Om$.  

For the Dirichlet-to-Neumann magnetic Schr\"{o}dinger inverse problem, a partial output result was first given by Dos Santos Ferreira, Kenig, Sj\"{o}strand, and Uhlmann in ~\cite{DKSjU}.  This work was expanded to include a partial input result by this author in ~\cite{Ch1}.  Both of these papers build on previous work by Kenig, Sj\"{o}strand and Uhlmann in ~\cite{KSU}, which gives partial input and output results for the Calder\'{o}n problem.  These results for the Calder\'{o}n problem, as well as partial data results by Isakov ~\cite{I} have since been extended by Kenig and Salo in ~\cite{KeSa}.  Partial data results in unbounded domains have been achieved as well, for example in ~\cite{KLU} and ~\cite{Poh}.  For a more complete survey on recent partial data results, see ~\cite{KeSa_survey}.  

However, in the partial data case, the Dirichlet-to-Neumann and Neumann-to-Dirichlet problems are no longer equivalent -- the partial data problems for the Dirichlet-to-Neumann map represent different subsets of the Cauchy data than the partial data problems for the Neumann-to-Dirichlet map.  

The main result of this paper is to prove partial input and partial output theorems for the Neumann-to-Dirichlet magnetic Schr\"{o}dinger inverse problem, analogous to the ones in ~\cite{KSU}.  This is an extension of previous work in ~\cite{Ch2}, which proves a similar result for the Neumann-to-Dirichlet version of the Calder\'{o}n problem.  In order to describe the result properly, we will define `front' and `back' sets of $\Om$ as follows.  Suppose $\ph(x)$ is a smooth function on a neighbourhood of $\Om$.  Then we define $\partial \Om_{+}$ and $\partial \Om_{-}$ relative to $\ph$ by 
\begin{equation*}
\begin{split}
\partial \Om_{+} &= \{ p \in \partial \Om | \partial_{\nu} \ph(p) \geq 0 \} \\
\partial \Om_{-} &= \{ p \in \partial \Om | \partial_{\nu} \ph(p) \leq 0 \} \\
\end{split}
\end{equation*}

\begin{theorem}\label{MainThm}
Let $q_1, q_2 \in L^{\infty}(\Om)$, let $A_1, A_2$ be $C^2$ vector fields on $\Om$, with $\nu \cdot A_1 = \nu \cdot A_2$ on $\partial \Om$.  Define $\partial \Om_{+}$ and $\partial \Om_{-}$ using the function $\ph(x) = \pm \log |x-p|$, where $p \in \RnI$ is a point outside the closure of the convex hull of $\Om$.  Let $\Gamma_{\pm} \subset \partial \Om$ be neighbourhoods of $\partial \Om_{\pm}$.  Suppose
\[
N_{A_1,q_1} g|_{\Gamma_{+}} = N_{A_2,q_2} g|_{\Gamma_{+}}
\]
for all $g \in H^{-\half}(\partial \Om)$ whose support is contained in $\Gamma_{-}$.  Then $dA_1 = dA_2$, and $q_1 = q_2$.  
\end{theorem}

Note that if $\Om$ is strictly convex, then Theorem \ref{MainThm}, with the choice of $\ph = -\log|x-p|$, implies that the set on which the Neumann-Dirichlet maps are measured can be made arbitrarily small, by proper choice of $p$.  On the other hand, choosing $\ph = +\log|x-p|$ would imply that the set on which the input functions are supported can be arbitrarily small.  

The main new ingredient used in the proof of Theorem \ref{MainThm} is the following Carleman estimate, which allows us to construct $H^1$ complex geometric optics (CGO) solutions for the problem $\Laq u = 0$.  

\begin{theorem}\label{MainCarl}
Suppose $\ph$ and $\Gamma_{+}$ are as in Theorem \ref{MainThm}, and $w \in H^1(\Om)$ is such that
\begin{equation}\label{logBC}
\begin{split}
w, \partial_{\nu} w &= 0 \mbox{ on } \Gamma_{+} \\
h\partial_{\nu} w &= (\partial_{\nu} \ph) w + h\s w \mbox{ on } \Gamma_{+}^c \\
\end{split}
\end{equation}
for some order zero operator $\s$ bounded uniformly in $h$.  There exists $h_0 > 0$ such that if $0<h<h_0$, then
\begin{equation}\label{theCarl}
h^{\half} \|w\|_{L^2(\Gamma_{+}^c)} + h\|w\|_{L^2(\Om)} \lesssim \| \Lphaq w \|_{H^{1*}(\Om)}.
\end{equation}
Here $H^{1*}$ is the dual space to the semiclassical $H^1$ space with semiclassical parameter $h$, and $\Lphaq$ is the conjugated operator
\[
\Lphaq = h^2 e^{\frac{\ph}{h}}\Laq e^{-\frac{\ph}{h}}.
\]   
\end{theorem}

All Sobolev norms here and in the rest of the paper are semiclassical unless otherwise stated. The complex geometrical optics solutions are then as follows.

\begin{prop}\label{CGOs}
Let $\ph$ and $\Gamma_{-}$ be as in Theorem \ref{MainThm}.  Then there exists a solution $u \in H^1(\Om)$ of the problem 
\begin{eqnarray*}
\Laq u &=& 0 \mbox{ on } \Om\\
\nu \cdot (\grad + iA) u|_{\Gamma_{-}^c} &=& 0 
\end{eqnarray*}
of the form $u = e^{\frac{1}{h}(-\ph + i \psi)}(a+r)$, where $a$ is a $C^2$ function with bounds uniform in $h$; $\psi$ is a smooth real solution to the eikonal equation $\grad \ph \cdot \grad \psi = 0,  |\grad \psi| = |\grad \ph|$; and $\|r\|_{H^1(\Om)} \leq O(h^{\half})$.  In particular, $\ph$, $\psi$ and $a$ are as in the CGO solutions in ~\cite{DKSjU}.
\end{prop}

It is worth pausing here to compare the Carleman estimate in Theorem \ref{MainCarl} to the Carleman estimate from Theorem 1.3 in ~\cite{Ch2}, which concludes that if $w \in H^1(\Om)$, and satisfies \eqref{logBC}, then
\[
h^{\half} \|w\|_{H^1(\Gamma_{+}^c)} + h\|w\|_{H^1(\Om)} \lesssim \| \Lphaq w \|_{L^{2}(\Om)}.
\]
In Theorem \ref{MainCarl}, all of the Sobolev norms have essentially been shifted down by one.  This shift is what allows us to create the $H^1$ CGO solutions from Proposition \ref{CGOs}, which in turn are critical for handling the first order term in the operator \eqref{Laq}.  

The proof of Theorem \ref{MainCarl} is the main new technical contribution of this paper.  The key point is the construction of continuous operators from $H^1$ to $L^2$ and vice versa, which preserve the boundary conditions \eqref{logBC}, and have the commutator properties of a semiclassical pseudodifferential operator.  This allows us to obtain Theorem \ref{MainCarl} from Theorem 1.3 of ~\cite{Ch2} by an appropriate substitution, and careful analysis of the resulting error terms.  The construction is a fairly delicate matter, since it requires splitting the function $w$ into small and large frequency parts, and creating the appropriate operator for each part.  In order to present the proof of Theorem \ref{MainCarl} clearly, it will help to first describe the proof of a modified version, where $\ph$ is linear instead of logarithmic.

\begin{theorem}\label{LinearCarl}
Suppose $\ph(x) = \lambda \cdot x$, where $\lambda$ is a fixed unit vector in $\RnI$, and $\Gamma_{+}$ is a neighbourhood of $\partial \Om_{+}$.  Let $w \in H^1(\Om)$ be such that
\begin{equation}\label{BC}
\begin{split}
w, \partial_{\nu} w &= 0 \mbox{ on } \Gamma_{+} \\
h\partial_{\nu} (e^{-\frac{\ph}{h}} w) &= h\s e^{-\frac{\ph}{h}} w \mbox{ on } \Gamma_{+}^c \\
\end{split}
\end{equation}
for some smooth function $\s$ bounded uniformly in $h$.  There exists $h_0 > 0$ such that if $0<h<h_0$, then
\begin{equation}\label{thelinCarl}
h^{\half} \|w\|_{L^2(\Gamma_{+}^c)} + h\|w\|_{L^2(\Om)} \lesssim \| \Lphaq w \|_{H^{1*}(\Om)},
\end{equation}
where $\Lphaq$ is the conjugated operator
\[
\Lphaq = h^2 e^{\frac{\ph}{h}}\Laq e^{-\frac{\ph}{h}}.
\]   
\end{theorem}

The plan of this paper is then as follows.  In the next section we will prove Theorem \ref{MainThm} using Theorem \ref{MainCarl} and Proposition \ref{CGOs}. In Section 3 we will introduce modified versions of the operators in ~\cite{Ch2}, and in Section 4, we will use these to prove the Carleman estimate Theorem \ref{LinearCarl}.  In Section 5, we will modify these arguments to deal with the logarithmic $\ph$, and thus prove Theorem \ref{MainCarl}. Finally, Proposition \ref{CGOs} will be proved in Section 6. 
\vspace{4mm}

\noindent \textbf{Acknowledgements}  This work was partly done at the University of Jyv\"{a}skyl\"{a}, with support from the Academy of Finland.  The author is also very grateful to Mikko Salo for his time and support, and for many helpful conversations and comments during his time in Jyv\"{a}skyl\"{a}.  

\section{Proof of Theorem \ref{MainThm}}

Suppose that 
\[
u_1 = e^{\frac{-\ph + i\psi}{h}}(a_1 + r_1)
\]
is a CGO solution to
\begin{eqnarray*}
\L_{A_1, q_1} u_1 &=& 0 \mbox{ on } \Om\\
\nu \cdot (\grad + iA) u_1 |_{\Gamma_{-}^c} &=& 0, 
\end{eqnarray*}
as obtained from Proposition \ref{CGOs}.  Let 
\[
u_2 = e^{\frac{\ph + i\psi}{h}}(a_2 + r_2)
\]
be a standard CGO solution to $\L_{A_2,\overline{q}_2} u_2 = 0$, with no known conditions on its boundary behaviour.  

Now define $w \in H^1(\Om)$ to be the solution to
\begin{equation*}
\begin{split}
\L_{A_2,q_2} w &= 0 \mbox{ on } \Om \\
\nu \cdot (\grad + iA_2) w|_{\partial \Om} &= \nu \cdot (\grad + iA_2) u_1 |_{\partial \Om} 
\end{split}
\end{equation*}.

Then consider the integral
\[
\int_{\partial \Om} (N_{A_1,q_1} - N_{A_2, q_2})(\nu \cdot (\grad + iA_1)u_1)\overline{\nu \cdot (\grad + iA_2)u_2} dS.
\]
By definition of $u_1$, and the assumption on $N_{A_1,q_1}$ and  $N_{A_2, q_2}$, this is
\[
\int_{\Gamma_{+}^c} (N_{A_1,q_1} - N_{A_2, q_2})(\nu \cdot (\grad + iA_1)u_1))\overline{\nu \cdot (\grad + iA_2)u_2} dS.
\]
Now $u_1$ has been chosen so that $N_{A_1, q_1}(\nu \cdot (\grad + iA_1)u_1)) = u_1$ on $\Gamma_+^c$.  Similarly, since $\nu \cdot A_1 = \nu \cdot A_2$, we have that $N_{A_2, q_2}(\nu \cdot (\grad + iA_1)u_1)) = w$ on $\Gamma_+^c$.  Therefore we get
\[
\int_{\Gamma_{+}^c} (u_1 - w)\overline{\nu \cdot (\grad + iA_2)u_2} dS.
\]
Now by Green's theorem
\[
\int_{\Gamma_{+}^c} (u_1 - w)\overline{\nu \cdot (\grad + iA_2)u_2} dS = \int_{\Om} (u_1 - w) \L_{A_2,q_2}\overline{u}_2 dV - \int_{\Om} \L_{A_2,q_2}(u_1 - w)\overline{u_2} dV.
\]
The other boundary term vanishes since $\partial_{\nu} w|_{\partial \Om} = \partial_{\nu} u_1 |_{\partial \Om}$.  Moreover, the first term on the right side is zero by definition of $u_2$, so
\begin{equation*}
\begin{split}
\int_{\Gamma_{+}^c} (u_1 - w)\overline{\nu \cdot (\grad + iA_2)u}_2 dS &= \int_{\Om} \L_{A_2,q_2}(w - u_1)\overline{u_2} dV \\ 
                                                                                           &= \int_{\Om} (\L_{A_1,q_1} - \L_{A_2,q_2})(u_1) \overline{u_2} dV. \\
\end{split}
\end{equation*}
Therefore
\begin{equation}\label{FinalGreen}
\begin{split}
& \int_{\Gamma_{+}^c} (u_1 - w)\overline{\nu \cdot (\grad + iA_2)u}_2 dS = \\
& \int_{\Om} (A_1^2 - A_2^2 + q_1 - q_2)u_1 \overline{u_2} dV + \int_{\Om} (A_1 - A_2)\cdot(u_1\overline{Du_2}+ D u_1 \overline{u_2})dV.
\end{split}
\end{equation}
Now as in ~\cite{DKSjU} or ~\cite{Ch2}, the integral on the left can be bounded by 
\[
h^{\half}\|e^{\frac{\ph}{h}}(u_1 - w)\|_{L^2(\Gamma_{+}^c)} \cdot  h^{-\half} \left(\|e^{-\frac{\ph}{h}}\partial_{\nu} u_2\|_{L^2(\Gamma_{+}^c)} + \|e^{-\frac{\ph}{h}}u_2\|_{L^2(\Gamma_{+}^c)}\right),
\] 
and
\[
 \left(\|e^{-\frac{\ph}{h}}\partial_{\nu} u_2\|_{L^2(\Gamma_{+}^c)} + \|e^{-\frac{\ph}{h}}u_2\|_{L^2(\Gamma_{+}^c)}\right) = O(h^{-\frac{3}{2}}).
\]
Meanwhile, $e^{\frac{\ph}{h}}(u_1 - w)$ satisfies \eqref{BC}, so by Theorem \ref{MainCarl}, 
\begin{equation*}
\begin{split}
h^{\half}\|e^{\frac{\ph}{h}}(u_1 - w)\|_{L^2(\Gamma_{+}^c)} &\lesssim \| \L_{A_2,q_2,\ph} e^{\frac{\ph}{h}}(u_1 - w) \|_{H^{1*}(\Om)} \\
                                                                     &=        h^2 \| e^{\frac{\ph}{h}}(\L_{A_2,q_2} - \L_{A_1,q_1})u_1 \|_{H^{1*}(\Om)}, \\
\end{split}
\end{equation*} 
and the last line can be expanded as
\[
h^2 \| e^{\frac{\ph}{h}}(A_2^2 - A_1^2 +q_1 - q_2 + (A_1-A_2)\cdot D + D\cdot (A_2-A_1))u_1 \|_{H^{1*}(\Om)}.
\]

Since $u_1 = e^{\frac{-\ph + i\psi_1}{h}}(a_1 + r_1)$, this is $O(h)$.  Therefore
\[
\left| \int_{\Gamma^c} (u_1 -w) \overline{\nu \cdot (\grad + iA_2)u}_2 dS \right| = O(h^{-\half}).
\]
By using the explicit forms of $u_1$ and $u_2$, we can see that the first term on the right side of \eqref{FinalGreen} is $O(1)$.  Therefore multiplying \eqref{FinalGreen} by $h$ and taking the limit as $h \rightarrow 0$ gives
\[
\lim_{h \rightarrow 0} \int_{\Om} (A_1 - A_2)\cdot(u_1 h\overline{Du_2}+ hD u_1 \overline{u_2})dV = 0.
\]
Expanding using the expressions for $u_1$ and $u_2$, and applying the conditions on $a_1, a_2, r_1$ and $r_2$, we get
\[
\lim_{h \rightarrow 0} \int_{\Om} (A_1 - A_2) \cdot (\grad \ph - i\grad \psi) a_1 \overline{a_2} dV = 0.
\]
Now we are in the position of ~\cite{DKSjU}, and it follows by the arguments there that $dA_1 = dA_2$.  Then by a gauge transform we can assume $A_1 = A_2$, and so \eqref{FinalGreen} becomes
\[
\int_{\Gamma_{+}^c} (u_1 - w)\overline{\nu \cdot (\grad + iA_2)u}_2 dS = \int_{\Om} (q_1 - q_2)u_1 \overline{u_2} dV.
\]
Now the left side integral is $O(h^{\half})$, so we get
\[
\lim_{h \rightarrow 0} \int_{\Om} (q_1 - q_2)u_1 \overline{u_2} dV = 0,
\]
and as in ~\cite{DKSjU} we have enough information to conclude that $q_1 = q_2$ .  

\section{Operators}

We will now turn to the proof of Theorem \ref{LinearCarl}.  To begin, we will introduce the operators $J, J^{*}, J^{-1},$ and $J^{*-1}$.  Choose coordinates $(x,y)$ on $\RnI$ such that $x \in \Rn$ and $y \in \R$. Let $\RnIp$ denote the set $\{(x,y) \in \RnI| y > 0 \}$, and let $\Rno$ denote the boundary of $\RnIp$.  Let $\S(\RnIp)$ denote the set of restrictions to $\RnIp$ of Schwartz functions on $\RnI$.  Then for $u \in \S(\RnIp)$, let $\hat{u}$ indicate the semiclassical Fourier transform in the $x$-variables only.  Now let $F$ be a complex-valued function such that 
\begin{equation}\label{FBehaviour}
F(\xi), \mathrm{Re} F(\xi) \simeq  1 + |\xi|,
\end{equation}
and define the operators $J, J^{*}, J^{-1},$ and $J^{*-1}$ by
\begin{eqnarray*}
\widehat{Ju}(\xi,y)       &=& (F(\xi) + h\partial_y)\hat{u}(\xi,y) \\
\widehat{J^*u}(\xi,y)     &=& (\overline{F(\xi)} - h\partial_y)\hat{u}(\xi,y) \\
\widehat{J^{-1}u}(\xi,y)  &=& \frac{1}{h}\int_{0}^y \hat{u}(\xi,t) e^{F(\xi)\frac{t-y}{h}}dt \mbox{ and } \\
\widehat{J^{*-1}u}(\xi,y) &=& \frac{1}{h}\int_y^\infty \hat{u}(\xi,t) e^{\overline{F(\xi)}\frac{y-t}{h}}dt.
\end{eqnarray*}
These operators have the following boundedness properties.    

\begin{lemma}\label{bddness}
For $u \in \S(\RnIp)$,
\begin{eqnarray*}
\|Ju\|_{L^2(\RnIp)}       &\lesssim& \|u\|_{H^1(\RnIp)} \\
\|J^{*}u\|_{L^2(\RnIp)}   &\simeq&   \|u\|_{H^1(\RnIp)} \\
\|J^{-1}u\|_{H^1(\RnIp)}  &\simeq&   \|u\|_{L^2(\RnIp)} \mbox{ and }\\
\|J^{*-1}u\|_{H^1(\RnIp)} &\simeq&   \|u\|_{L^2(\RnIp)}. \\
\end{eqnarray*}
Moreover, if $u \in \S(\RnIp)$, and $u(x,0) = 0$ for all $x$, then
\[
\|Ju\|_{L^2(\RnIp)} \simeq \|u\|_{H^1(\RnIp)}.
\]
\end{lemma}

\begin{proof}
These inequalities follow from Lemma 5.1 in ~\cite{Ch2}, along with the facts that 
\[
J J^{-1} u = u
\]
for all $u \in \S(\RnIp)$ and
\[
J^{-1} J u = u
\]
for all $u \in \S(\RnIp)$ such that $u = 0$ at $y= 0$.

\end{proof}

In addition, if $F$ is a symbol of first order, so
\begin{equation}\label{FSymbol}
|\partial_{\xi}^{\alpha} F(\xi)| \lesssim C_{\alpha} (1 + |\xi|)^{1 - |\a|},
\end{equation}
then the operators defined above map $\S(\RnIp)$ to itself, and we have the following commutator properties, from Lemma 5.2 in ~\cite{Ch1}.
\begin{lemma}\label{CommProps}
Suppose $w \in \S(\RnIp)$ and $\chi \in \S(\RnIp)$.  Then
\[
\| J \chi w \|_{L^2(\RnIp)} \gtrsim \| \chi J w\|_{L^2(\RnIp)} - h\| w\|_{L^2(\RnIp)}.
\]
The constant in the $\gtrsim$ sign depends on $F$ and $\chi$, but not $h$.
\end{lemma}

The proof requires the following operator fact, which we'll record here.  Let $m,k \in \mathbb{Z}$, with $m,k\geq 0$.  Suppose $a(x,\xi,y)$ are smooth functions on $\Rn \times \Rn \times \R$ that satisfy the bounds
\[
|\partial_x^\b \partial_\xi^{\a} \partial_y^j a(x,\xi,y)| \leq C_{\a,\b} (1 + |\xi|)^{m-|\a|}
\]
for all multiindices $\a$ and $\b$, and for $0 \leq j \leq k$.  In other words, each $\partial_y^j a(x,\xi,y)$ is a symbol on $\Rn$ of order $m$, with bounds uniform in $y$, for $0 \leq j \leq k$.  Then we can define an operator $A$ on Schwartz functions in $\RnI$ by applying the pseudodifferential operator on $\Rn$ with symbol $a(x,\xi,y)$, defined by the Kohn-Nirenberg quantization, to $f(x,y)$ for each fixed $y$.  More generally, we can also define operators $A_j$ on Schwartz functions in $\RnI$ by applying the pseudodifferential operator on $\Rn$ with symbol $\partial_y^j a(x,\xi,y)$ to $f(x,y)$ for each fixed $y$, for $1 \leq j \leq k$.  Then Lemma 5.2 from ~\cite{Ch2} is as follows.

\begin{lemma}\label{flatops}
If $A$ is as above, then $A$ extends to a bounded operator from $H^{k+m}(\RnI)$ to $H^k(\RnI)$.
\end{lemma}

Now for something original.  Suppose $F$ satisfies \eqref{FBehaviour}, and define the operator $P$ by 
\[
\widehat{Pu}(\xi,y) = \hat{u}(\xi,0)e^{-\frac{F(\xi)y}{h}}
\]
for $u \in \S(\RnIp)$.  Then $P$ maps $\S(\RnIp)$ to itself, and a simple integral calculation shows that 
\begin{equation}\label{PL2bound}
\|Pu\|_{L^2(\RnIp)} \lesssim h^{\half} \|u\|_{H^{-\half}(\Rno)}
\end{equation}
and
\begin{equation}\label{PH1bound}
\|Pu\|_{H^1(\RnIp)} \simeq h^{\half} \|u\|_{H^{\half}(\Rno)} \lesssim \|u\|_{H^1(\RnIp)}.
\end{equation}
Moreover, note that $JPu = 0$. 

We have the following lemma.  

\begin{lemma}\label{H1star}
Suppose $u \in \S(\RnIp)$.  Let $H^{1*}(\RnIp)$ be the dual space of $H^1(\RnIp)$.  Then
\[
\|J^{*}u\|_{H^{1*}(\RnIp)} \simeq \|u\|_{L^2(\RnIp)} + h^{\half} \|u\|_{H^{-\half}(\Rno)}.
\]
In addition, if $E$ is any first order differential operator or $E = J$, then
\[
\|Eu\|_{H^{1*}(\RnIp)} \lesssim \|u\|_{L^2(\RnIp)} + h^{\half}\|u\|_{H^{-\half}(\Rno)}.
\]
\end{lemma}

\begin{proof}
First,
\begin{equation}\label{DualityCalc}
\begin{split}
\|J^{*}u\|_{H^{1*}(\RnIp)} &= \sup_{v \in H^1(\RnIp) \neq 0} \frac{|(J^{*}u,v)|}{\|v\|_{H^1(\RnIp)}} \\
                           &= \sup_{v \in H^1(\RnIp) \neq 0} \frac{|(u,Jv) + h(u,v)_{\Rno}|}{\|v\|_{H^1(\RnIp)}} \\
\end{split}                          
\end{equation}
Now choosing $v = J^{-1} u$ gives
\[
\|J^{*}u\|_{H^{1*}(\RnIp)} \gtrsim \frac{\|u\|_{L^2(\RnIp)}^2}{\|J^{-1} u\|_{H^1(\RnIp)}},
\]
with the boundary term vanishing, since $J^{-1} u = 0$ on the boundary.  Using the boundedness properties now gives us 
\[
\|J^{*}u\|_{H^{1*}(\RnIp)} \gtrsim \|u\|_{L^2(\RnIp)}.
\]
On the other hand, if $u|_{\Rno} \neq 0$, we can go back to \eqref{DualityCalc} and choose $v$ defined by
\[
\hat{v} = \frac{1}{F(\xi)}\widehat{Pu}.
\]
Using the notation $T_{\psi}$ to denote the operator defined by the Fourier multiplier $\psi$, 
\[
v = T_{F^{-1}}Pu = PT_{F^{-1}}u.
\]
Now
\[
\|J^{*}u\|_{H^{1*}(\RnIp)} \gtrsim \frac{h\|T_{F^{-\half}}u\|^2_{L^2(\Rno)}}{\|PT_{F^{-1}}u\|_{H^1(\RnIp)}};
\]
here the non-boundary term disappears since $JP = 0$.  Now the boundedness properties of $P$ and $T_{F^{-\half}}$ give us
\begin{eqnarray*}
\|J^{*}u\|_{H^{1*}(\RnIp)} &\gtrsim& \frac{h\|u\|^2_{H^{-\half}(\Rno)}}{h^{\half}\|T_{F^{-1}}u\|_{H^{\half}(\Rno)}} \\
                           &\simeq&  h^{\half}\|u\|_{H^{-\half}(\Rno)}.
\end{eqnarray*}
Now if $E = J,J^{*},$ or any other first order differential operator, then the argument used in \eqref{DualityCalc} gives us that 
\begin{eqnarray*}
\|Eu\|_{H^{1*}(\RnIp)} &\leq&     \sup_{v \in H^1(\RnIp) \neq 0} \frac{|(u,Ev) + Ch(u,v)_{\Rno}|}{\|v\|_{H^1(\RnIp)}} \\
                       &\lesssim& \|u\|_{L^2(\RnIp)} + h\frac{\|u\|_{H^{-\half}(\Rno)}\|v\|_{H^{\half}(\Rno)}}{\|v\|_{H^1(\RnIp)}} \\
                       &\lesssim& \|u\|_{L^2(\RnIp)} + h^{\half}\|u\|_{H^{-\half}(\Rno)}. \\
\end{eqnarray*}

\end{proof}

\section{Proof of Theorem \ref{LinearCarl}}

This section is devoted to the proof of Theorem \ref{LinearCarl}.  For the rest of this section, $\ph$ is assumed to be as in the statement of that theorem.  Now we may as well choose coordinates $(x,y)$ as in the previous section, so $x \in \Rn$, $y \in \R$, and $\ph(x,y) = y$.

To prove the Carleman estimate, we will need to work with the convexified Carleman weights, as in ~\cite{DKSjU} and ~\cite{Ch1}.  Let 
\[
\ph_c = \ph + \frac{h}{2\e}\ph^2,
\]
and define 
\[
\Lphe = h^2 e^{\frac{\ph_c}{h}}\Lap e^{-\frac{\ph_c}{h}}
\]
and
\[
\Lphaqe = h^2 e^{\frac{\ph_c}{h}}\Laq e^{-\frac{\ph_c}{h}}.
\]
Now suppose $\Om_2 $ is a smooth bounded domain in $\RnIp$ such that $\Om \subset \subset \Om_2$ as subsets of $\RnIp$, and $\Gamma_{+}^c \subset \partial \Om_2 $.  Let $\Gamma_{2+} \subset \partial \Om_2$ such that $\Gamma_{+}^c \subset \Gamma_{2+}^c$.  We have the following proposition rewritten from Theorem 1.3 of ~\cite{Ch2}.

\begin{prop}\label{NDCarl}
Suppose $w \in H^2(\Om_2)$, such that 
\begin{equation}\label{BC2}
\begin{split}
w, \partial_{\nu} w &= 0 \mbox{ on } \Gamma_{2+} \\
(h\partial_{\nu} - \partial_{\nu} \ph)w  &= h\sigma w \mbox{ on } \Gamma_{2+}^c. \\
\end{split}
\end{equation}
where $\sigma$ is an order zero operator bounded uniformly in $h$.  Then
\[
h^{\half}\|w\|_{H^1(\Gamma_{2+}^c)} + \frac{h}{\sqrt{\varepsilon}}\|w\|_{H^1(\Om_2)} \lesssim \| \Lphe w\|_{L^2(\Om_2)}.
\]
\end{prop}
This will be the starting point for the proof of Theorem \ref{LinearCarl}.  It is essentially the estimate we want, but we need to shift the indices down in each Sobolev space that appears in the estimate, without disturbing the boundary term.  

\subsection{The Flat Case}

To illustrate the idea of the proof, we will first sketch the proof in the case where $\Gamma_{+}^c$ lies in the plane $y = 0$.  Then the boundary conditions on $\Gamma_{+}^c$ become
\begin{equation}\label{FlatBC}
h\partial_{y}w = w + h\sigma w.
\end{equation}
Suppose $F$ satisfies \eqref{FBehaviour} and \eqref{FSymbol}, and define $J$, $P$, and the related operators as in Section 3.  Now if $w$ satisfies \eqref{FlatBC}, then the first thing to notice is that $(J^{-1} + T_{(1 + F(\xi))^{-1}}P)w$ does as well.  If we were to define $Q$ to be the operator $(J^{-1} + T_{(1 + F(\xi))^{-1}}P)$, then $\chi Qw$ satisfies \eqref{BC2} for some appropriate cutoff function $\chi$.  By applying Proposition \ref{NDCarl} to $\chi Qw$, and use the commutator and boundedness results from Section 3, we can get 
\[
h^{\half}\|w\|_{L^2(\Rno)} + \frac{h}{\sqrt{\varepsilon}}\|w\|_{L^2(\RnIp)} \lesssim \| \Lphe Q w\|_{L^2(\RnIp)}.
\]
If it were generally true that 
\begin{equation}\label{hopesndreams}
\|v\|_{L^2(\RnIp)} \lesssim \|Jv\|_{H^{1*}(\RnIp)},
\end{equation}
then we would have
\[
h^{\half}\|w\|_{L^2(\Rno)} + \frac{h}{\sqrt{\varepsilon}}\|w\|_{L^2(\RnIp)} \lesssim \|J\Lphe Q w\|_{H^{1*}(\RnIp)}.
\]
Then by using commutator properties for $J$, together with the fact that $JQw = w$, we would have that 
\[
h^{\half}\|w\|_{L^2(\Rno)} + \frac{h}{\sqrt{\varepsilon}}\|w\|_{L^2(\RnIp)} \lesssim \|\Lphe w\|_{H^{1*}(\RnIp)},
\]
and from here the remainder of the proof would be simple.  Unfortunately \eqref{hopesndreams} is not true in general.  However, as in ~\cite{Ch1}, we can hope to prove that it holds for $v$ of the form $\Lphe Q w$.  The reason for this is that $\Lph$ factors as
\[
\Lph = (h\partial_y - T_{1 + |\xi|})(h\partial_y - T_{1 - |\xi|}).
\]
Thus if $F$ is chosen well, $\Lphe$ can be factored as $J^{*}B$, where $B$ looks like $h\partial_y - T_{1 - |\xi|}$, up to appropriate error.  Now for $v$ of the form $J^* BQw$,
\[
\|Jv\|_{H^{1*}(\RnIp)} = \|JJ^* BQw \|_{H^{1*}(\RnIp)}. 
\]
$J$ and $J^*$ commute, and so by Lemma \ref{H1star},
\[
\|Jv\|_{H^{1*}(\RnIp)} \gtrsim \|JBQw \|_{L^2(\RnIp)}.
\]
Now since $JP = 0$, we can write this as 
\[
\|Jv\|_{H^{1*}(\RnIp)} \gtrsim \|J(BQw - PBQw) \|_{L^2(\RnIp)}.
\]
Then $BQw - PBQw = 0$ on $\Rno$, so by Lemma \ref{bddness}, 
\begin{eqnarray*}
\|Jv\|_{H^{1*}(\RnIp)} &\gtrsim& \|(BQw - PBQw) \|_{H^1(\RnIp)} \\
                           &\gtrsim& \|BQw\|_{H^1(\RnIp)} - \|PBQw\|_{H^{1}(\RnIp)} \\
                           &\simeq&  \|J^* BQw\|_{L^2(\RnIp)} - h^{\half}\|BQw\|_{H^{\half}(\Rno)}. \\
\end{eqnarray*}
Therefore
\begin{equation}\label{WithBndryTerm}                        
\|Jv\|_{H^{1*}(\RnIp)} =  \|v\|_{L^2(\RnIp)} - h^{\half}\|BQw\|_{H^{\half}(\Rno)}. \\
\end{equation}
On the other hand, by Lemma 3.4 
\[
\|Jv\|_{H^{1*}(\RnIp)}  = \|J^* JBQw \|_{H^{1*}(\RnIp)} \gtrsim h^{\half} \|JBQw\|_{H^{-\half}(\Rno)}.
\]
Up to acceptable error, $J$ and $B$ commute, and $JQ$ is the identity.  Therefore
\[
\|Jv\|_{H^{1*}(\RnIp)} \gtrsim h^{\half} \|Bw\|_{H^{-\half}(\Rno)}.
\]
Now one can check that if $B = h\partial_y - T_{1 - |\xi|}$, and $w$ satisfies the boundary condition \eqref{FlatBC}, then 
\[
h^{\half} \|Bw\|_{H^{-\half}(\Rno)} \simeq h^{\half}\|BQw\|_{H^{\half}(\Rno)}.
\]
Combining this with the previous inequality and substituting into \eqref{WithBndryTerm} shows that \eqref{hopesndreams} holds in this case.  This finishes the sketch of the proof.  In reality, everything is much more complicated.  To begin, we will need to do a change of variables to be able to work with a flat boundary.  This changes $\Lphe$ and the boundary condition somewhat.  Then the factoring becomes much more complicated, and as in ~\cite{Ch1}, we will have to break things into a small frequency case and a large frequency case and prove things separately in each case.     

\subsection{A Graph Case}
We will begin the proof of Theorem \ref{LinearCarl} by considering the special case in which $\Gamma_{+}^c$ coincides with a graph of the form $y = f(x)$ where $f$ is a smooth function with some constant vector $K \in \Rn$ such that $|\grad f - K| \leq \delta$, for some small $\delta > 0$ to be chosen later.  Then we can ask that $\Gamma_{2+}^c$ satisfies the same graph conditions.  

In this case we'll do a change of variables $(x,y) \mapsto (x, y - f(x))$ to flatten out the graph.  

Let $\tilde{\Om}_2$ and $\tilde{\Gamma}_{2+}$ be the images of $\Om_2$ and $\Gamma_{2+}$ respectively, under this map.  Note that $\Gamma_{2+} \subset \Rno$. Then we have the following proposition.

\begin{prop}\label{CofVCarl}
Suppose $w \in H^2(\tilde{\Om}_2)$, and  
\begin{equation}\label{tildeBC}
\begin{split}
w, \partial_{\nu} w &= 0 \mbox{ on } \tilde{\Gamma}_{2+} \\
h\partial_{y} w   &= \frac{w + \grad f \cdot h\grad w}{1 + |\grad f|^2} +h\sigma w \mbox{ on } \tilde{\Gamma}_{2+}^c. \\
\end{split}
\end{equation}
Then
\begin{equation*}
h^{\half} \|w\|_{H^1(\tilde{\Gamma}_{2+}^c)} + \frac{h}{\sqrt{\e}}\|w\|_{H^1(\tilde{\Om}_2)} \lesssim \| \Lphet w \|_{L^2(\tilde{\Om}_2)},
\end{equation*}
where
\[
\Lphet = (1+|\grad f|^2)h^2\partial_y^2 - 2(\a + \grad f \cdot h\grad_x)h\partial_y + \a^2 + h^2\Lap_x
\]
and $\a = (1 + \frac{h}{\e}(y + f(x))$.
\end{prop}

\begin{proof}
Suppose $w \in H^2(\Omt)$ satisfies \eqref{tildeBC}.  Let $v$ be the function on $\Om$ defined by $v(x,y) = w(x, y + f(x))$.  Then $v$ satisfies \eqref{BC2}, and thus
\[
h^{\half} \|v\|_{H^1(\Gamma_{2+}^c)} + \frac{h}{\sqrt{\e}}\|v\|_{H^1(\Om_2)} \lesssim \| \Lphe v \|_{L^2(\Om_2)}.
\]
By a change of variables, $\|v\|_{H^1(\Om_2)} \simeq \|w\|_{H^1(\tilde{\Om}_2)}$ and $\|v\|_{H^1(\Gamma_{2+}^c)} \simeq \|w\|_{H^1(\tilde{\Gamma}_{2+}^c)}$. Moreover, 
\[
\left( \Lphe v \right)(x, y - f(x)) = \Lphet \left( w(x,y) \right) + hE_1 w(x,y)
\]
where $E_1$ is a first order semiclassical differential operator.  Thus by another change of variables,
\[
\|\Lphe v \|_{L^2(\Om_2)} \lesssim \|\Lphet w \|_{L^2(\Omt_2)} + h\|w\|_{H^1(\Omt_2)}.
\]
Putting this together gives
\[
h^{\half} \|w\|_{H^1(\tilde{\Gamma}_{2+}^c)} + \frac{h}{\sqrt{\e}}\|w\|_{H^1(\tilde{\Om}_2)} \lesssim \| \Lphet w \|_{L^2(\tilde{\Om}_2)}+h\|w\|_{H^1(\Omt_2)},
\]
and the last term on the right side can be absorbed into the left side to finish the proof.
\end{proof}

Now having changed variables, we will shift Sobolev spaces in the Carleman estimate above.

\begin{prop}\label{FlatHMinus1}
Suppose $w \in H^1(\tilde{\Om})$ and $|\grad f - K| \leq \delta$, and 
\begin{equation}\label{tildeBC2}
\begin{split}
w, \partial_{\nu} w &= 0 \mbox{ on } \tilde{\Gamma}_{2+} \\
h\partial_{y} w   &= \frac{w + \grad f \cdot h\grad w}{1 + |\grad f|^2} + h\sigma w \mbox{ on } \tilde{\Gamma}_{2+}^c. \\
\end{split}
\end{equation}
For sufficiently small $\delta$, there exists $h_0 > 0$ such that if $0<h<h_0$, then
\begin{equation}\label{FlatHMinus1Carl}
h^{\half} \|w\|_{L^2(\tilde{\Gamma}_{+}^c)} + \frac{h}{\sqrt{\e}}\|w\|_{L^2(\tilde{\Om})} \lesssim \| \Lphet w \|_{H^{1*}(\tilde{\Om})}.
\end{equation}
\end{prop}  

Suppose $w \in \S(\RnIp)$ has support inside $\tilde{\Om}$ and satisfies the boundary conditions \eqref{tildeBC2}.  We want to split $w$ into low and high frequency parts and prove the inequality separately for each of them.  Thus, choose $m_2 > m_1 > 0$, and $\mu_1$ and $\mu_2$ such that 
\[
\frac{|K|}{\sqrt{1+|K|^2}} < \mu_1 < \mu_2 < \half + \frac{|K|}{2\sqrt{1+|K|^2}} < 1.
\]
The eventual choice of $\mu_j$ and $m_j$ will depend only on $K$ and the constant from the Carleman estimate in Proposition \ref{CofVCarl}.  

Define $\rho \in \Czinf(\Rn)$ such that $\rho(\xi) = 1$ if $|\xi| < \mu_1$ and $|K \cdot \xi| < m_1$, and $\rho(\xi) = 0$ if $|\xi| > \mu_2$ or $|K \cdot \xi| > m_2$.  Define $w_s = T_{\rho} w$ and $w_{\ell} = (1 - T_{\rho}) w$, so $w = w_s + w_{\ell}$.  We will prove the following lemmas.

\begin{lemma}\label{smalllemma}
Suppose $w \in \S(\RnIp)$ has support inside $\tilde{\Om}$ and satisfies the boundary conditions \eqref{tildeBC2}, and let $w_{s}$ be defined as above.  Then for appropriate choices of $\delta$, $\mu_1, \mu_2, m_1,$ and $m_2$, 
\begin{equation*}
h^{\half} \|w_{s}\|_{L^2(\Rno)} + \frac{h}{\sqrt{\e}}\|w_{s}\|_{L^2(\RnIp)} \lesssim \| \Lphet w_{s} \|_{H^{1*}(\RnIp)} + h\|w\|_{L^2(\RnIp)} + h\|w\|_{L^2(\Rno)}.
\end{equation*}
\end{lemma} 

\begin{lemma}\label{largelemma}
Suppose $w \in \S(\RnIp)$ has support inside $\tilde{\Om}$ and satisfies the boundary conditions \eqref{tildeBC2}, and let $w_{\ell}$ be defined as above.  Then for appropriate choice of $\delta$, 
\begin{equation*}
h^{\half} \|w_{\ell}\|_{L^2(\Rno)} + \frac{h}{\sqrt{\e}}\|w_{\ell}\|_{L^2(\RnIp)} \lesssim \| \Lphet w_{\ell} \|_{H^{1*}(\RnIp)}+h\|w\|_{L^2(\RnIp)}+ h\|w\|_{L^2(\Rno)}.
\end{equation*}
\end{lemma} 

\begin{proof}[Proof of Proposition \ref{FlatHMinus1}]
The proof of Proposition \ref{FlatHMinus1} follows from these, since we can add up the inequalities to get
\[
h^{\half} \|w\|_{L^2(\Rno)}+\frac{h}{\sqrt{\e}}\|w\|_{L^2(\RnIp)} \lesssim \| \Lphet w_{s} \|_{H^{1*}(\RnIp)}+ \| \Lphet w_{\ell} \|_{H^{1*}(\RnIp)}.
\]
Now $\Lphet w_{s} = \Lphet T_{\rho} w$, so using the commutator properties and Lemma \ref{H1star}, 
\begin{eqnarray*}
& & \|\Lphet w_{s}\|_{H^{1*}(\RnIp)} \simeq \|(1+|\grad f|^2)^{-1} \Lphet T_{\rho} w\|_{H^{1*}(\RnIp)} \\
&\lesssim&   \|T_{\rho}(1+|\grad f|^2)^{-1} \Lphet w\|_{H^{1*}(\RnIp)} + h\|h\partial_y w\|_{H^{1*}(\RnIp)} + h\|w\|_{L^2(\RnIp)} \\
&\lesssim& \|\Lphet w \|_{H^{1*}(\RnIp)} + h\|w\|_{L^2(\RnIp)} + h^{\frac{3}{2}}\|w\|_{H^{-\half}(\Rno)} \\
\end{eqnarray*}
A similar inequality holds for $\Lphet w_{\ell}$.  Therefore for $w \in \S(\RnIp)$ with support inside $\tilde{\Om}$ satisfying \eqref{tildeBC2},
\[
h^{\half} \|w\|_{L^2(\Rno)}+\frac{h}{\sqrt{\e}}\|w\|_{L^2(\RnIp)} \lesssim \| \Lphet w \|_{H^{1*}(\RnIp)}+ h\|w\|_{L^2(\RnIp)}+ h\|w\|_{L^2(\Rno)},
\]
and we can absorb the last two terms into the left side to get 
\[
h^{\half} \|w\|_{L^2(\Rno)}+\frac{h}{\sqrt{\e}}\|w\|_{L^2(\RnIp)} \lesssim \| \Lphet w\|_{H^{1*}(\RnIp)}.
\]
Since $w$ is supported in $\tilde{\Om}$, and $w$ vanishes to first order on $\tilde{\Gamma}_{+}$, this shows \eqref{FlatHMinus1Carl} holds for $w \in \S(\RnIp)$ which satisfy the conditions of the proposition.  Then Proposition \ref{FlatHMinus1} follows from a density argument.  
\end{proof}

\begin{proof}[Proof of Lemma \ref{smalllemma}]
To begin, define 
\[
A_{\pm}(V,\xi) = \frac{1 + iV \cdot \xi \pm \sqrt{(1+iV\cdot \xi)^2 - (1+|V|^2)(1-|\xi|^2)}}{1 + |V|^2},
\]
so $A_{\pm}(V,\xi)$ are roots of the polynomial
\[
(1 + |V|^2)X^2 - 2(1+ iV \cdot \xi)X + (1 - |\xi|^2).
\]
Here we'll choose the branch of the square root with non-negative imaginary part, so the branch cut occurs along the positive real axis.  Note that $A(K,\xi)$ is smooth as a function of $\xi$ except when the argument of the square root lies on the branch cut; i.e., when $K \cdot \xi = 0$ and $1 - (1+|K|^2)(1-|\xi|^2)$ is positive.  There $A(K,\xi)$ has a jump discontinuity of size $1 - (1+|K|^2)(1-|\xi|^2)$.  Now for $\delta_2 > 0$, if $\mu_2$ is small enough, then we can arrange for 
\[
1 - (1+|K|^2)(1-|\xi|^2) \leq \d_2
\]
on the support of $\rho(\xi)$.  Then we can choose a smooth function $F(\xi)$ such that  on the support of $\rho(\xi)$,
\[
|F(\xi) - A(K,\xi)| \leq \delta_2.
\]
Now consider the bounds on $A_{\pm}(K,\xi)$ on the support of $\rho(\xi)$.  By our choice of $\mu_2$, on the support of $\hat{w}_s$, the expression
\[
(1+iK\cdot \xi)^2 - (1+|K|^2)(1-|\xi|^2)
\]
has real part confined to the interval $[-K^2-m_2^2, \delta_2+m_2^2]$, and imaginary part confined to the interval $[-2m_2, 2m_2]$.  Therefore, if $\delta_2$ and $m_2$, are small enough, we can ensure
\[
\mathrm{Re}A_{\pm}(K, \xi) > \frac{1}{2(1+|K|^2)}.
\]
on the support of $\rho$.  Therefore we can take our choice of $F$ to satisfy \eqref{FBehaviour} and \eqref{FSymbol}, so if we define $J$, $J^{-1}$, $J^{*}$, $J^{*-1}$, and $P$ with this choice of $F$, they have all of the properties listed in the previous section.  This allows us to fix the choice of $\mu_1, \mu_2, m_1,$ and $m_2$, depending only on $K$ and $\delta_2$.  

Let $\chi \in C^{\infty}(\RnIp)$ be a cutoff function which is identically one on $\tilde{\Om}$ and identically zero outside $\tilde{\Om}_2$, with $\chi \equiv 1$ on $\tilde{\Gamma}_{2+}^c$ and $\partial_y \chi = 0$ at $y = 0$. Note that 
\[
F(\xi) + \frac{1 + i\grad f \cdot \xi}{1 + |\grad f|^2}
\]
has real part greater than $(1+|K|^2)^{-1}$, so 
\[
\Phi = \left( F(\xi) + \frac{1 + i\grad f \cdot \xi}{1 + |\grad f|^2} \right)^{-1}
\]
is a smooth bounded function in $x$ and $\xi$.  Moreover, one can check that $\Phi$ satisfies the conditions to be a symbol of order $-1$ on $\Rno$.  Now if $w_s$ is as in the statement of the theorem, define
\[
Qw_s = (J^{-1} + T_{\Phi}P)w_s,
\]
By combining the boundedness results for $J^{-1}$, $P$, and $T_{\Phi}$, we get
\begin{equation}\label{QH1Bound}
\|Qv\|_{H^1(\RnIp)} \lesssim \|v\|_{L^2(\RnIp)} + h^{\half}\|v\|_{H^{-\half}(\Rno)}
\end{equation}
and
\begin{equation}\label{QHHalfRnoBound}
\|Qv\|_{H^{\half}(\Rno)} \lesssim \|v\|_{H^{-\half}(\Rno)}.
\end{equation}
for $v \in \S(\RnIp)$.

Now consider the function $\chi Qw_s$.  This is a smooth function on $\tilde{\Om}_2$, and by definition of $\chi$, it vanishes to first order at $\tilde{\Gamma}_{2+}$. Meanwhile, on $\tilde{\Gamma}^c_{2+}$, $\chi Qw = \chi (J^{-1} + T_{\Phi}P)w_s = T_{\Phi}w_s$, and
\[
h \partial_y \chi Qw = h \chi \partial_y (J^{-1} + T_{\Phi}P)w_s = -F(\xi)J^{-1}w_s + w_s - T_{\Phi}T_F w_s
\]
Since $J^{-1}w_s = 0$ at $y= 0$, 
\begin{eqnarray*}
h \partial_y (\chi Q w_s)|_{y = 0} &=& w_s|_{y = 0} - T_{\Phi}T_F w_s|_{y = 0} \\
                                   &=& T_{1 -\Phi F} w_s|_{y = 0} + hE_0 w_s|_{y = 0} \\
                                   &=& \frac{1 + \grad f \cdot h\grad}{1 + |\grad f|^2}T_{\Phi} w_s|_{y = 0} + hE_0 w_s|_{y = 0} \\
                                   &=& \frac{1 + \grad f \cdot h\grad}{1 + |\grad f|^2}(\chi Q w_s)|_{y = 0} + hE_0 w_s|_{y = 0}. \\
\end{eqnarray*}
Therefore $\chi Q w_s$ satisfies \eqref{tildeBC}.  Then by Proposition \ref{CofVCarl}, 
\begin{equation}\label{SubCarl}
h^{\half} \|\chi Qw_s\|_{H^1(\tilde{\Gamma}_{2+}^c)} + \frac{h}{\sqrt{\e}}\|\chi Qw_s\|_{H^1(\tilde{\Om}_2)} \lesssim \| \Lphet \chi Qw_s \|_{L^2(\tilde{\Om}_2)}.
\end{equation}

We will obtain the desired Carleman estimate from this by a series of claims.  Our first task is to remove some of the $Q$s from \eqref{SubCarl}.  Note that we can afford to let errors bounded by $h\|w\|_{L^2(\RnIp)}+ h^{\frac{3}{2}}\|w\|_{L^2(\Rno)}$ accumulate, since a term of this size appears on the right side of the estimate in Lemma \ref{smalllemma}.  Therefore let $R$ denote the expression $h\|w\|_{L^2(\RnIp)}+ h^{\frac{3}{2}}\|w\|_{L^2(\Rno)}$.  Then the first claim is as follows.  
\vspace{3mm}

\noindent \textbf{Claim S.1} 
\begin{equation}\label{PenCarl}
h^{\half}\| w_s\|_{L^2(\Rno)}+\frac{h}{\sqrt{\e}}\|w_s\|_{L^2(\RnIp)} \lesssim \| \Lphet Qw_s \|_{L^2(\RnIp)}+ R
\end{equation}

\vspace{3mm}

To prove Claim S.1, we will consider the terms of \eqref{SubCarl} one by one.  First, 
\begin{eqnarray*}
h^{\half} \|\chi Q w_s\|_{H^1(\tilde{\Gamma}_{2+}^c)} &=& h^{\half} \|\chi (J^{-1} + T_{\Phi}P) T_{\rho} w\|_{H^1(\tilde{\Gamma}_{2+}^c)}\\
                                                      &=& h^{\half} \|\chi T_{\Phi}T_{\rho} w\|_{H^1(\tilde{\Gamma}_{2+}^c)} \\
\end{eqnarray*}
since $J^{-1}w_s|_{y=0} = 0$, and $Pw_s|_{y=0} = w_s|_{y=0}$.  Now using commutator properties for pseudodifferential operators, together with Lemma \ref{flatops}, gives 
\begin{eqnarray*}
\|\chi T_{\Phi}T_{\rho} w\|_{H^1(\Rno)} &\gtrsim& \|T_{\Phi}T_{\rho} \chi w\|_{H^1(\Rno)} - h\|w\|_{L^2(\Rno)} \\
                                        &\gtrsim& \| T_{\Phi}w_s\|_{H^1(\Rno)} - h\|w\|_{L^2(\Rno)}, \\
\end{eqnarray*}
so substituting this into the previous inequality and invoking the boundedness properties of $T_{\Phi}$ gives 
\begin{equation}\label{FirstCarlTerm}
h^{\half} \|\chi Q w_s\|_{H^1(\tilde{\Gamma}_{2+}^c)} \gtrsim h^{\half}\| w_s\|_{L^2(\Rno)} - h^{\frac{3}{2}}\|w\|_{L^2(\Rno)}. 
\end{equation}
Secondly,
\[
\|\chi Q w_s\|_{H^1(\tilde{\Om}_2)}  \gtrsim \|J\chi Q w_s\|_{L^2(\RnIp)} 
\]
by Lemma \ref{bddness}.  Then by Lemma \ref{CommProps},
\begin{eqnarray*}
\|J \chi Q w_s\|_{L^2(\RnIp)} &\gtrsim& \|\chi J(J^{-1} + T_{\Phi}P)w_s\|_{L^2(\RnIp)} - h\|Qw_s\|_{L^2(\RnIp)} \\
                              &\gtrsim& \|\chi w_s + JT_{\Phi}Pw_s\|_{L^2(\RnIp)} - h\|Qw_s\|_{L^2(\RnIp)}. \\
\end{eqnarray*}
Using \eqref{QH1Bound},
\[
\|J \chi Q w_s\|_{L^2(\RnIp)} \gtrsim \|\chi w_s + JT_{\Phi}Pw_s\|_{L^2(\RnIp)} - R.
\]
Now $JT_{\Phi}Pw_s = T_{\Phi}JPw_s + hE_{-1}Pw_s =  hE_{-1}Pw_s$, where $E_{-1}$ is an operator of order $-1$ in the $x$ variables, so
\begin{eqnarray*}
\|J \chi Q w_s\|_{L^2(\RnIp)} &\gtrsim& \|\chi w_s \|_{L^2(\RnIp)}-h\|Pw_s\|_{L^2(\RnIp)}-R \\
                              &\gtrsim& \|\chi w_s\|_{L^2(\RnIp)} - R \\
\end{eqnarray*}                
Meanwhile
\[
\|\chi w_s\|_{L^2(\RnIp)} = \|\chi T_{\rho} w\|_{L^2(\RnIp)} \gtrsim \| w_s\|_{L^2(\RnIp)} - h\|w\|_{L^2(\RnIp)},
\]
so
\begin{equation}\label{SecondCarlTerm}
 \|J \chi Q w_s\|_{L^2(\RnIp)} \gtrsim \|w_s\|_{L^2(\RnIp)} - R.
\end{equation}
Thirdly, 
\begin{eqnarray*}
\| \Lphet \chi Qw_s \|_{L^2(\RnIp)} &\lesssim& \| \Lphet Q w_s \|_{L^2(\RnIp)} + h\|Qw_s \|_{H^1(\RnIp)} \\
                                    &\lesssim& \| \Lphet Qw_s \|_{L^2(\RnIp)}+ R. \\
\end{eqnarray*}

Substituting this expression, along with \eqref{FirstCarlTerm} and \eqref{SecondCarlTerm} back into \eqref{SubCarl} and absorbing away extraneous terms as necessary establishes \eqref{PenCarl} and thus proves the first claim.  The next step is to figure out a way to replace the $\| \Lphet Qw_s \|_{L^2(\RnIp)}$ term in \eqref{PenCarl} with $\|\Lphet w_s\|_{H^{1*}(\RnIp)}$.  This leads us to the second claim.

\vspace{3mm}
\noindent \textbf{Claim S.2}
\[
\| \Lphet Qw_s\|_{L^2(\RnIp)} \lesssim \|\Lphet w_s\|_{H^{1*}(\RnIp)}+ h^{\half}\|J^{*-1} \Lphet Qw_s\|_{H^{\half}(\Rno)} + R.
\]
\vspace{3mm}

To prove Claim S.2, first consider $\| J \Lphet Qw_s \|_{H^{1*}(\RnIp)}$.  By expanding $\Lphet$ and using Lemma \ref{CommProps},
\begin{eqnarray*}
\| J \Lphet Qw_s \|_{H^{1*}(\RnIp)} &\lesssim& \| \Lphet JQw_s \|_{H^{1*}(\RnIp)} + h\| (h\partial_y E_1+ E_2) Qw_s \|_{H^{1*}(\RnIp)}. \\
                                    &\lesssim& \| \Lphet w_s \|_{H^{1*}(\RnIp)} + h\| (h\partial_y E_1+ E_2) Qw_s \|_{H^{1*}(\RnIp)}. \\
\end{eqnarray*}
where $E_1$ and $E_2$ are first and second order operators, respectively, in the $x$ variables.  Thus by Lemma \ref{H1star} and a similar calculation for the transversal operator $E_2$, 
\begin{eqnarray*}
& & \|J\Lphet Qw_s\|_{H^{1*}(\RnIp)} \\
&\lesssim& \|\Lphet w_s\|_{H^{1*}(\RnIp)}+h\|E_1 Qw_s\|_{L^2(\RnIp)}+h\|Qw_s\|_{H^1(\RnIp)}+h^{\frac{3}{2}}\|E_1 Qw_s\|_{H^{-\half}(\Rno)}. \\
\end{eqnarray*}
Therefore
\[
\|J\Lphet Qw_s\|_{H^{1*}(\RnIp)} \lesssim \|\Lphet w_s\|_{H^{1*}(\RnIp)}+h\|Qw_s\|_{H^1(\RnIp)}+h^{\frac{3}{2}}\|Qw_s\|_{H^{\half}(\Rno)}.
\]
Applying the bounds for $Q$ in equations \eqref{QH1Bound} and \eqref{QHHalfRnoBound} gives
\begin{equation}\label{S2toL2line}
\|J\Lphet Qw_s\|_{H^{1*}(\RnIp)} \lesssim \|\Lphet w_s\|_{H^{1*}(\RnIp)} + R.
\end{equation}
Now
\begin{eqnarray*}
\| J \Lphet Q w_s \|_{H^{1*}(\RnIp)} &=&       \| JJ^{*}J^{*-1} \Lphet Qw_s \|_{H^{1*}(\RnIp)} \\
                                     &=&       \| J^{*}JJ^{*-1} \Lphet Qw_s \|_{H^{1*}(\RnIp)} \\
                                     &\gtrsim& \| JJ^{*-1} \Lphet Qw_s \|_{L^2(\RnIp)} \\
\end{eqnarray*}
Since $JP = 0$,
\[
\| J \Lphet Q w_s \|_{H^{1*}(\RnIp)} \gtrsim \| J(J^{*-1} \Lphet Qw_s - PJ^{*-1} \Lphet Qw_s) \|_{L^2(\RnIp)}.
\]
Now $J^{*-1} \Lphet Qw_s - PJ^{*-1} \Lphet Qw_s$ is zero at $\Rno$, so applying the appropriate boundedness result for $J$ gives
\begin{eqnarray*}                                     
\| J \Lphet Q w_s \|_{H^{1*}(\RnIp)} &\gtrsim& \| J^{*-1} \Lphet Qw_s - PJ^{*-1} \Lphet Qw_s \|_{H^1(\RnIp)} \\
                                     &\gtrsim& \| \Lphet Qw_s\|_{L^2(\RnIp)} - \|PJ^{*-1} \Lphet Qw_s\|_{H^1(\RnIp)} \\
                                     &\gtrsim& \| \Lphet Qw_s\|_{L^2(\RnIp)}-h^{\half}\|J^{*-1} \Lphet Qw_s\|_{H^{\half}(\Rno)}. \\
\end{eqnarray*}
Substituting this into \eqref{S2toL2line} proves Claim S.2.  

The next goal is to eliminate the $h^{\half}\|J^{*-1} \Lphet Qw_s\|_{H^{\half}(\Rno)}$ in Claim S.2.  Before we get to this, though, we'll need an intermediate claim.

\vspace{3mm}
\noindent \textbf{Claim S.3}
\[
\|w_s\|_{H^1(\RnIp)} \lesssim \|\Lphet w_s\|_{H^{1*}(\RnIp)} + h^{\half}C_{\delta} \|w_s\|_{L^2(\Rno)}.
\]
\vspace{3mm}

To prove Claim S.3, we can expand $\Lphet$ and take advantage of the assumptions that $|\a-1|, |\grad f - K| \leq \delta$ to write
\begin{eqnarray*}
\|\Lphet w_s\|_{H^{1*}(\RnIp)} &\gtrsim& \|((1+|K|^2)h^2\partial_y^2 - 2(1 + K\cdot h\grad_x)h\partial_y + 1 +h^2\Lap_x)w_s\|_{H^{1*}(\RnIp)} \\
                               & &       - C_{\delta} \|E_2w_s\|_{H^{1*}(\RnIp)} \\
                               &\gtrsim& \|(h\partial_y-T_{A_-(K,\xi)})(h\partial_y - T_{A_+(K,\xi)})w_s\|_{H^{1*}(\RnIp)} \\
															 & & - C_\delta \|E_2w_s\|_{H^{1*}(\RnIp)} \\
\end{eqnarray*}
Now on the support of $\hat{w}_s$, both $A_{+}(K,\xi)$ and $A_{-}(K,\xi)$ have the necessary properties of $F$ to prove the boundedness properties from Section 3.  Therefore by Lemma \ref{H1star},
\begin{eqnarray*}
\|\Lphet w_s\|_{H^{1*}(\RnIp)} &\gtrsim& \|(h\partial_y -T_{A_+})w_s\|_{L^2(\RnIp)}- C_\delta\|E_2w_s\|_{H^{1*}(\RnIp)} \\
\end{eqnarray*}
Thus
\begin{equation*}
\|\Lphet w_s\|_{H^{1*}(\RnIp)} \gtrsim \|w_s\|_{H^1(\RnIp)} - C_\delta \|E_2 w_s\|_{H^{1*}(\RnIp)}. 
\end{equation*}
Meanwhile we can write 
\[
E_2 w_s = (E_1 E_1' + E_1'')w_s
\]
where $E_1, E_1',$ and $E_1''$ are all first order operators.  Then applying Lemma \ref{H1star}, 
\[
\|E_2 w_s\|_{H^{1*}(\RnIp)} \lesssim \|w_s\|_{H^1(\RnIp)}+ h^{\half}\|w_s\|_{H^{\half}(\Rno)}.
\]
Since $\hat{w}_s$ is supported only for small frequencies, $\|w_s\|_{H^{\half}(\Rno)} \simeq \|w_s\|_{H^{-\half}(\Rno)}$, so this proves Claim S.3.

Note that Claim S.3 is a small frequency phenomenon only -- it only works because $w_s$ vanishes at large frequencies.  In large frequencies, though, we'll be able to factor $\Lphet$ using pseudodifferential operators, which we cannot do at small frequencies because of smoothness issues with the resulting symbols.  

Now we are ready to deal with the $h^{\half}\|J^{*-1} \Lphet Qw_s\|_{H^{\half}(\Rno)}$ term from Claim S.2.

\vspace{3mm}
\noindent \textbf{Claim S.4}
\begin{equation}\label{ClaimS4}
h^{\half}\|J^{*-1} \Lphet Q w_s\|_{H^{\half}(\Rno)} \lesssim \|\Lphet w_s\|_{H^{1*}(\RnIp)} + h^{\half} C_{\delta, \delta_2} \|w_s\|_{L^2(\Rno)}.
\end{equation}
\vspace{3mm}

The first step in the proof of Claim S.4 is to note that by Lemma \ref{H1star}, 
\[
\| \Lphet w_s \|_{H^{1*}(\RnIp)} \gtrsim h^{\half}\| J^{*-1} \Lphet w_s \|_{H^{-\half}(\Rno)}.
\]
Therefore it will suffice to show that $\| J^{*-1} \Lphet w_s \|_{H^{-\half}(\Rno)}$ and $\|J^{*-1} \Lphet Q w_s\|_{H^{\half}(\Rno)}$ are comparable up to some acceptable error.

This means we need to calculate $J^{*-1} \Lphet w_s$ and $J^{*-1} \Lphet Q w_s$ at $y=0$.  These are nearly identical calculations, so we'll start with $J^{*-1} \Lphet w_s$.
\begin{eqnarray*}
\widehat{J^{*-1} \Lphet w_s}|_{y=0} &=& \frac{1}{h}\int_0^\infty \widehat{\Lphet w_s}(\xi, t)e^{\frac{-F(\xi)t}{h}}dt \\
                                    &=& \frac{1}{h}\int_0^\infty h^2\partial_t^2\mathcal{F}[(1+|\grad f|^2)w_s]e^{\frac{-F(\xi)t}{h}}dt \\
                                    & & - \frac{2}{h}\int_0^\infty h\partial_t\mathcal{F}[(\a + \grad f \cdot h\grad_x)w_s]e^{\frac{-F(\xi)t}{h}}dt \\
                                    & & + \frac{1}{h}\int_0^\infty (\widehat{\a^2 w_s} -|\xi|^2\hat{w}_s)e^{\frac{-F(\xi)t}{h}}dt \\
\end{eqnarray*}
We can integrate by parts to get rid of the $h\partial_t$'s:
\begin{eqnarray*}
\widehat{J^{*-1} \Lphet w_s}|_{y=0} &=& h^{-1}\int_0^\infty F^2(\xi)\mathcal{F}[(1+|\grad f|^2)w_s]e^{\frac{-F(\xi)t}{h}}dt \\
                                    & & - 2h^{-1}\int_0^\infty F(\xi)\mathcal{F}[(\a + \grad f \cdot h\grad_x)w_s]e^{\frac{-F(\xi)t}{h}}dt \\
                                    & & + h^{-1}\int_0^\infty (\widehat{\a^2 w_s} -|\xi|^2\hat{w}_s)e^{\frac{-F(\xi)t}{h}}dt \\
                                    & & -h\partial_t \mathcal{F}[(1+|\grad f|^2)w_s]|_{t=0}+ 2\mathcal{F}[(\a + \grad f \cdot h\grad_x)w_s]|_{t=0} \\
                                    & & -F(\xi)\mathcal{F}[(1+|\grad f|^2)w_s]|_{t=0}.
\end{eqnarray*}
Now we take advantage of the fact that $|\a-1|, |\grad f - K| \leq \delta$ to write
\begin{eqnarray*}
\widehat{J^{*-1} \Lphet w_s}|_{y=0} &=& h^{-1}\int_0^\infty F^2(\xi)\mathcal{F}[(1+|K|^2)w_s]e^{\frac{-F(\xi)t}{h}}dt \\
                                    & & - 2h^{-1}\int_0^\infty F(\xi)\mathcal{F}[(1 + K \cdot h\grad_x)w_s]e^{\frac{-F(\xi)t}{h}}dt \\
                                    & & + h^{-1}\int_0^\infty (1 - |\xi|^2)\widehat{w_s} e^{\frac{-F(\xi)t}{h}}dt + h^{-1}\int_0^\infty \delta \widehat{E_2 w}_s e^{\frac{-F(\xi)t}{h}}dt \\
                                    & & -h\partial_t \mathcal{F}[(1+|\grad f|^2)w_s]|_{t=0}+ 2\mathcal{F}[(\a + \grad f \cdot h\grad_x)w_s]|_{t=0} \\
                                    & & -F(\xi)\mathcal{F}[(1+|\grad f|^2)w_s]|_{t=0}, \\
\end{eqnarray*}
where $E_2$ is a second order operator in the $x$ variables with uniform bounds in $\delta$.  Then by definition of $F$, we in fact get
\begin{eqnarray*}
\widehat{J^{*-1} \Lphet w_s}|_{y=0} &=& h^{-1}\int_0^\infty (\delta + \delta_2) E_2 \hat{w}_s e^{\frac{-F(\xi)t}{h}}dt -h\partial_t \mathcal{F}[(1+|\grad f|^2)w_s]|_{t=0}\\
                                    & & + 2\mathcal{F}[(\a + \grad f \cdot h\grad_x)w_s]|_{t=0} -F(\xi)\mathcal{F}[(1+|\grad f|^2)w_s]|_{t=0}, \\
\end{eqnarray*}
where $E_2$ is a (different) second order operator in the $x$ variables with uniform bounds in $\delta$.  Now we invoke the boundary conditions.  Since $w$ satisfies the boundary condition
\[
(1 + |\grad f|^2) h\partial_{y} w = w + \grad f \cdot h\grad w + h\sigma w \mbox{ on } \tilde{\Gamma}_{2+}^c,
\]
$w_s$ satisfies the same boundary condition, albeit with a different $\sigma$.  Therefore we get
\begin{eqnarray*}
\widehat{J^{*-1} \Lphet w_s}|_{y=0} &=& h^{-1}\int_0^\infty (\delta + \delta_2) \widehat{E_2 w}_s e^{\frac{-F(\xi)t}{h}}dt + \mathcal{F}[(\a + \grad f \cdot h\grad_x)w_s]|_{t=0} \\
                                    & & -F(\xi)\mathcal{F}[(1+|\grad f|^2)w_s]|_{t=0} - h\widehat{\sigma w}_s|_{t=0}\\
                                    &=& h^{-1}\int_0^\infty (\delta + \delta_2) \widehat{E_2 w}_s e^{\frac{-F(\xi)t}{h}}dt \\
                                    & & + \mathcal{F}[(\a + \grad f \cdot h\grad_x - T_F(1+|\grad f|^2)) w_s]|_{t=0}- h\widehat{\sigma w}_s|_{t=0}.\\
\end{eqnarray*}
Now precisely the same calculation applies to $J^{*-1} \Lphet Q w_s$, so
\begin{eqnarray*}
\widehat{J^{*-1} \Lphet Q w_s}|_{y=0} &=& h^{-1}\int_0^\infty (\delta+\delta_2) \mathcal{F}[E_2 Q w_s] e^{\frac{-F(\xi)t}{h}}dt \\
                                    & & + \mathcal{F}[(\a + \grad f \cdot h\grad_x - T_F(1+|\grad f|^2))Q w_s]|_{t=0} - h\mathcal{F}[ \sigma Qw_s]|_{t=0}.\\
\end{eqnarray*}
At $t= 0$, $Q w_s = T_{\Phi}w_s$, so 
\begin{equation}\label{JLQw}
\begin{split}
\widehat{J^{*-1} \Lphet Q w_s}|_{y=0} =& h^{-1}\int_0^\infty (\delta+\delta_2) \mathcal{F}[E_2 Q w_s] e^{\frac{-F(\xi)t}{h}}dt \\
                                    &+ \mathcal{F}[(\a + \grad f \cdot h\grad_x - T_F(1+|\grad f|^2))T_{\Phi} w_s]|_{t=0}- h\widehat{\sigma T_{\Phi}w}_s|_{t=0}.\\
\end{split}
\end{equation}
We are interested in the quantity 
\[
\|J^{*-1} \Lphet Q w_s\|^2_{H^{\half}(\Rno)} \simeq h^{-n}\int_{\Rn} (1+|\xi|) |\mathcal{F}[J^{*-1} \Lphet Q w_s]|_{y=0}|^2 d\xi.
\]
Substituting the expression from \eqref{JLQw} and integrating, we get 
\begin{eqnarray*}
\|J^{*-1} \Lphet Q w_s\|^2_{H^{\half}(\Rno)} &\lesssim& \|(\a + \grad f \cdot h\grad_x - T_F(1+|\grad f|^2))T_{\Phi} w_s \|^2_{H^{\half}(\Rno)} \\
& &        + h^2\|w_s\|^2_{H^{-\half}(\Rno)} + h^{-1}C_{\delta,\delta_2} \|T_{(1+|\xi|)^{\half}}E_2 T_{\Phi}w_s\|^2_{L^2(\RnIp)} .\\
\end{eqnarray*}
Now since $\hat{w}_s$ is supported only for small $|\xi|$, 
\begin{eqnarray*}
\|J^{*-1} \Lphet Q w_s\|^2_{H^{\half}(\Rno)} &\lesssim& \|(\a + \grad f \cdot h\grad_x - T_F(1+|\grad f|^2))T_{\Phi} w_s \|^2_{H^{\half}(\Rno)} \\
& &        + h^2\|w_s\|^2_{H^{-\half}(\Rno)} + h^{-1}C_{\delta,\delta_2} \| w_s\|^2_{L^2(\RnIp)}. \\
\end{eqnarray*}
Now using commutator properties of pseudodifferential operators on $\Rno$, we get
\begin{eqnarray*}
\|J^{*-1} \Lphet Qw_s\|^2_{H^{\half}(\Rno)} &\lesssim& \|T_{\Phi}(\a + \grad f \cdot h\grad_x - T_F(1+|\grad f|^2)) w_s \|^2_{H^{\half}(\Rno)} \\
& &        + h^2\|w_s\|^2_{H^{-\half}(\Rno)} + h^{-1}C_{\delta,\delta_2} \| w_s\|^2_{L^2(\RnIp)}. \\
&\lesssim& \|(\a + \grad f \cdot h\grad_x - T_F(1+|\grad f|^2)) w_s \|^2_{H^{-\half}(\Rno)} \\
& &        + h^2\|w_s\|^2_{H^{-\half}(\Rno)} + h^{-1}C_{\delta,\delta_2} \| w_s\|^2_{L^2(\RnIp)}. \\
\end{eqnarray*}

Meanwhile a similar calculation for $\| J^{*-1} \Lphet w_s \|_{H^{-\half}(\Rno)}$ yields
\begin{eqnarray*}
\|J^{*-1} \Lphet w_s\|^2_{H^{-\half}(\Rno)} &\gtrsim&  \|(\a + \grad f \cdot h\grad_x - T_F(1+|\grad f|^2)) w_s \|^2_{H^{-\half}(\Rno)} \\
& &        - h^2\|w_s\|^2_{H^{-\half}(\Rno)} - h^{-1}C_{\delta,\delta_2} \| w_s\|^2_{L^2(\RnIp)}. \\
\end{eqnarray*}
Therefore
\begin{equation}\label{ABwsL2}
\begin{split}
& \|J^{*-1} \Lphet Qw_s\|_{H^{\half}(\Rno)} \\
\lesssim &\|J^{*-1} \Lphet w_s\|_{H^{-\half}(\Rno)} +h\|w_s\|_{H^{-\half}(\Rno)} +h^{-\half}C_{\delta,\delta_2} \|w_s\|_{L^2(\RnIp)}. \\
\end{split}
\end{equation}
Invoking Claim S.3 lets us bound this by 
\[
\|J^{*-1} \Lphet w_s\|_{H^{-\half}(\Rno)} +h^{\half}C_{\delta,\delta_2}\|w_s\|_{H^{-\half}(\Rno)} +h^{-\half}\|\Lphet w_s\|_{H^{1*}(\RnIp)}.
\]
Then using Lemma \ref{H1star}, 
\[
\| \Lphet w_s \|_{H^{1*}(\RnIp)} \gtrsim h^{\half}\| J^{*-1} \Lphet w_s \|_{H^{-\half}(\Rno)},
\]
so we get
\[
h^{\half}\|J^{*-1} \Lphet Qw_s\|_{H^{\half}(\Rno)} \lesssim \|\Lphet w_s\|_{H^{1*}(\RnIp)}+h^{\half}C_{\delta,\delta_2} \|w_s\|_{H^{-\half}(\Rno)},
\]
which finishes the proof of Claim S.4.  
\vspace{3mm}

Now we can finish the proof of Lemma \ref{smalllemma}.  Substituting Claim S.2 into Claim S.1 gives
\begin{eqnarray*}
& &        h^{\half}\| w_s\|_{L^2(\Rno)}+\frac{h}{\sqrt{\e}}\|w_s\|_{L^2(\RnIp)} \\
&\lesssim& \|\Lphet w_s\|_{H^{1*}(\RnIp)}+ R + h^{\half}\|J^{*-1} \Lphet Qw_s\|_{H^{\half}(\Rno)}\\
\end{eqnarray*}
Then substituting Claim S.4 into this inequality and writing out $R$ in full gives
\begin{eqnarray*}
& &        h^{\half}\| w_s\|_{L^2(\Rno)}+\frac{h}{\sqrt{\e}}\|w_s\|_{L^2(\RnIp)} \\
&\lesssim& \|\Lphet w_s\|_{H^{1*}(\RnIp)}+ h\|w\|_{L^2(\RnIp)}+ h^{\frac{3}{2}}\|w\|_{L^2(\Rno)} +h^{\half} C_{\delta, \delta_2}\|w_s\|_{H^{-\half}(\Rno)}.\\
\end{eqnarray*}

Absorbing the last term on the right side into the left side finishes the proof.  Note $\delta_2$ depends only on the constant in the Carleman estimate from Proposition \ref{CofVCarl} and operator norms of $J$ and the related operators, which depend only on $K$.  This justifies the claim made in defining $m_i$ and $\mu_i$.  

\end{proof}

\begin{proof}[Proof of Lemma \ref{largelemma}]

To begin, redefine 
\[
A_{\pm}(V,\xi) = \frac{1 + iV \cdot \xi \pm \sqrt{(1+iV\cdot \xi)^2 - (1+|V|^2)(1-|\xi|^2)}}{1 + |V|^2},
\]
so $A_{\pm}(V,\xi)$ are roots of the polynomial
\[
(1 + |V|^2)X^2 - 2(1+ iV \cdot \xi)X + (1 - |\xi|^2),
\]
as before, but now take the branch of the square root with nonnegative real part, so the branch cut lies on the nonpositive real axis.  

Now define
\[
A^{\e}_{\pm}(V,\xi) = \frac{\a + iV \cdot \xi \pm \sqrt{(\a+iV\cdot \xi)^2 - (1+|V|^2)(\a^2-|\xi|^2)}}{1 + |V|^2},
\]
so $A^{\e}_{\pm}(V,\xi)$ are the roots of the polynomial
\[
(1 + |V|^2)X^2 - 2(\a+ iV \cdot \xi)X + (\a^2 - |\xi|^2),
\]
using the same branch of the square root as above.  (Recall that $\a$ is defined by $\a = 1 + \frac{h}{\e}(y + f(x))$.)  

Now set $\zeta \in C^{\infty}_0(\Rn)$ to be a smooth cutoff function such that $\zeta(\xi) \equiv 1$ if 
\[
|K \cdot \xi| < \half m_1 \mbox{ and } |\xi| < \half \frac{|K|}{\sqrt{1+|K|^2}} +  \half \mu_1, 
\]
and $\zeta \equiv 0$ if $|K \cdot \xi| \geq m_1$ or $|\xi| \geq \mu_1$.  Define 
\[
G_{\pm}(V, \xi) = (1-\zeta)A_{\pm}(V,\xi) + \zeta
\]
and
\[
G^{\e}_{\pm}(V, \xi) = (1-\zeta)A^{\e}_{\pm}(V,\xi) + \zeta.
\]

Consider the singular support of $A^{\e}_{\pm}(K,\xi)$.  These are smooth as functions of $x$ and $\xi$ except when the argument of the square root falls on the non-positive real axis.  This occurs when $K \cdot \xi = 0$ and
\[
|\xi|^2 \leq \frac{\a^2|K|^2}{1 + |K|^2}.
\] 
This does not occur on the support of $1 - \z$, so it follows that $G_{\pm}(K, \xi)$ are smooth, and one can check that they are symbols of first order on $\Rn$.  Moreover $G_{+}(K,\xi)$ satisfies \eqref{FBehaviour}, so we can now redefine $F(\xi) = G_{+}(K,\xi)$ and define $J$, $P$, $\Phi$ and the related operators with respect to this choice of $F$.  

Note that for $\d$ sufficiently small, depending on $K$, it's also true that 
\[
|\xi|^2 \leq \frac{\a^2|K|^2}{1 + |K|^2}
\] 
does not occur on the support of $1 -\z$.  Therefore
\[
G^{\e}_{\pm}(\grad f, \xi) = (1-\zeta)A^{\e}_{\pm}(\grad f,\xi) + \zeta  
\]
are smooth, and one can check that they are symbols of first order on $\Rn$.  

Now define
\[
Qw_{\ell} = (J^{-1} + T_{\Phi}P)w_{\ell}.
\]
This $Q$ has the same boundedness properties as the one from the small frequency case.  Moreover, consider the function $\chi Q w_{\ell}$, where $\chi$ is as in the proof of Lemma \ref{smalllemma}.  As before, this satisfies $\eqref{tildeBC}$, so by Proposition \ref{CofVCarl}, 
\begin{equation*}
h^{\half} \|\chi Qw_{\ell}\|_{H^1(\tilde{\Gamma}_{2+}^c)} + \frac{h}{\sqrt{\e}}\|\chi Qw_{\ell}\|_{H^1(\tilde{\Om}_2)} \lesssim \| \Lphet \chi Q w_{\ell} \|_{L^2(\tilde{\Om}_2)}.
\end{equation*}
By following the arguments in Claim S.1 from the small frequency case, this becomes

\vspace{3mm}
\noindent \textbf{Claim L.1}
\begin{equation}\label{PenCarlLarge}
h^{\half}\| w_{\ell}\|_{L^2(\Rno)}+\frac{h}{\sqrt{\e}}\|w_{\ell}\|_{L^2(\RnIp)} \lesssim \| \Lphet Q w_{\ell} \|_{L^2(\RnIp)} + R. 
\end{equation}
\vspace{3mm}

Here 
\[
R = h\|w\|_{L^2(\RnIp)}+ h^{\frac{3}{2}}\|w\|_{L^2(\Rno)}
\]
as before.  Now we want to replace the $\| \Lphet Q w_{\ell} \|_{L^2(\RnIp)}$ term on the right with $\| \Lphet w_{\ell} \|_{H^{1*}(\RnIp)}$.  As in the small frequency case, our first attempt at this involves an extra boundary term.  

\vspace{3mm}
\noindent \textbf{Claim L.2} 
\begin{eqnarray*}
\| \Lphet Q w_{\ell} \|_{L^2(\RnIp)} &\lesssim& \| \Lphet w_{\ell} \|_{H^{1*}(\RnIp)}+ R \\
                                     & &        + h^{\half}\|(h\partial_y-T_{G^{\e}_{-}(\grad f, \xi)}) Qw_{\ell}\|_{H^{\half}(\Rno)}. \\
\end{eqnarray*}
\vspace{3mm}

To prove Claim L.2, note first that
\begin{equation}\label{JLphetQ_Lphet}
\| J \Lphet Qw_{\ell} \|_{H^{1*}(\RnIp)}  \lesssim \| \Lphet w_{\ell} \|_{H^{1*}(\RnIp)}+ R
\end{equation}
by the same arguments used in the small frequency case to prove \eqref{S2toL2line}.  Therefore it suffices to show that 
\begin{eqnarray*}
\| \Lphet Q w_{\ell} \|_{L^2(\RnIp)} &\lesssim& \| J \Lphet Qw_{\ell} \|_{H^{1*}(\RnIp)}+R \\
                                     & &        + h^{\half}\|(h\partial_y-T_{G^{\e}_{-}(\grad f, \xi)}) Qw_{\ell}\|_{H^{\half}(\Rno)}. \\
\end{eqnarray*}

So let's first examine $\| J \Lphet Q w_{\ell} \|_{H^{1*}(\RnIp)}$. Using properties of pseudodifferential operators, and writing $A_{\pm}^\e$ for $A_{\pm}^\e(\grad f, \xi)$, we can write
\begin{eqnarray*}
& & (1 + |\grad f|^2)(h\partial_y - T_{G^{\e}_{+}(\grad f, \xi)})(h\partial_y - T_{G^{\e}_{-}(\grad f, \xi)}) \\
&=& (1 + |\grad f|^2)(h^2\partial_y^2 - T_{A^{\e}_{+} + A^{\e}_{-}}T_{1 - \zeta}h\partial_y +T_{A_{+}^{\e}A_{-}^{\e}}T_{1-\zeta}^2) \\
& & + (1 + |\grad f|^2)(h\partial_y T_{2\zeta} + T_{\zeta^2} + T_{1-\zeta}T_{A^{\e}_{+} + A^{\e}_{-}}T_{\zeta}) + hE_1 \\
\end{eqnarray*}
Since $\zeta = 0$ on the support of $\hat{w}_{\ell}$, 
\[
T_{\zeta}Qw_{\ell} = T_{\zeta}(J^{-1} + T_{\Phi}P)w_{\ell} = hE_{-1} P w_{\ell}.
\]
Here $E_1$ and $E_{-1}$ are some operators of order $1$ and $-1$, respectively. Then
\begin{eqnarray*}
& &       \| J \Lphet Qw_{\ell} \|_{H^{1*}(\RnIp)} \\
&\gtrsim& \| J (h\partial_y - T_{G^{\e}_{+}(\grad f, \xi)})(h\partial_y - T_{G^{\e}_{-}(\grad f, \xi)}) Qw_{\ell} \|_{H^{1*}(\RnIp)} \\
& &       - h\|J E_1Qw_{\ell} \|_{H^{1*}(\RnIp)} - h\|J E_0 Pw_{\ell} \|_{H^{1*}(\RnIp)}. \\
\end{eqnarray*}
Using Lemma 3.4, we see that the last two terms are bounded by $R$.  Now since $G^{\e}_{-}(\grad f, \xi)$ differs from $G_{+}(K,\xi)$ by $O(\delta)$, 
\begin{equation}\label{L4StartPt}
\begin{split}
\| J \Lphet Qw_{\ell} \|_{H^{1*}(\RnIp)} \gtrsim & \|J J^{*}(h\partial_y - T_{G^{\e}_{-}(\grad f, \xi)}) Qw_{\ell} \|_{H^{1*}(\RnIp)} \\
                                                 & - C_{\delta}\|JE_1'(h\partial_y-T_{G^{\e}_{-}(\grad f,\xi)})Qw_{\ell}\|_{H^{1*}(\RnIp)} - R. \\
\end{split}
\end{equation}
$J$ and $J^{*}$ commute, so using Lemma \ref{H1star} gives us
\begin{eqnarray*}
\| J \Lphet Qw_{\ell} \|_{H^{1*}(\RnIp)} &\gtrsim& \|J(h\partial_y - T_{G^{\e}_{-}(\grad f, \xi)}) Qw_{\ell} \|_{L^2(\RnIp)} \\
                                         & &       - C_{\delta}\|JE_1'(h\partial_y-T_{G^{\e}_{-}(\grad f,\xi)})Qw_{\ell}\|_{H^{1*}(\RnIp)} -R. \\
\end{eqnarray*}
Now $JP = 0$, so
\begin{eqnarray*}
\| J \Lphet Qw_{\ell} \|_{H^{1*}(\RnIp)} &\gtrsim& \|J((h\partial_y-T_{G^{\e}_{-}(\grad f, \xi)}) Qw_{\ell}-P(h\partial_y-T_{G^{\e}_{-}(\grad f, \xi)}) Qw_{\ell})\|_{L^2(\RnIp)} \\
& &       - C_{\delta}\|JE_1'(h\partial_y-T_{G^{\e}_{-}(\grad f,\xi)})Qw_{\ell}\|_{H^{1*}(\RnIp)} - R. \\
\end{eqnarray*}
Then $v - Pv = 0$ at $\Rno$, by definition of $P$.  Therefore we can use Lemma \ref{bddness} to show that 
\begin{eqnarray*}
& &       \| J \Lphet Qw_{\ell} \|_{H^{1*}(\RnIp)} \\
&\gtrsim& \|(h\partial_y-T_{G^{\e}_{-}(\grad f, \xi)}) Qw_{\ell}-P(h\partial_y-T_{G^{\e}_{-}(\grad f, \xi)}) Qw_{\ell}\|_{H^1(\RnIp)} \\
& &       - C_{\delta}\|JE_1'(h\partial_y-T_{G^{\e}_{-}(\grad f,\xi)})Qw_{\ell}\|_{H^{1*}(\RnIp)} - R \\
&\gtrsim& \|(h\partial_y-T_{G^{\e}_{-}(\grad f, \xi)}) Qw_{\ell}\|_{H^1(\RnIp)} -\|P(h\partial_y-T_{G^{\e}_{-}(\grad f, \xi)}) Qw_{\ell}\|_{H^1(\RnIp)} \\
& &       - C_{\delta}\|JE_1'(h\partial_y-T_{G^{\e}_{-}(\grad f,\xi)})Qw_{\ell}\|_{H^{1*}(\RnIp)} - R \\
&\gtrsim& \|(h\partial_y-T_{G^{\e}_{-}(\grad f, \xi)}) Qw_{\ell}\|_{H^1(\RnIp)} - h^{\half}\|(h\partial_y-T_{G^{\e}_{-}(\grad f, \xi)}) Qw_{\ell}\|_{H^{\half}(\Rno)} \\
& &       - C_{\delta}\|JE_1'(h\partial_y-T_{G^{\e}_{-}(\grad f,\xi)})Qw_{\ell}\|_{H^{1*}(\RnIp)} - R. \\
\end{eqnarray*}
Now
\begin{eqnarray*}
& &       \|(h\partial_y-T_{G^{\e}_{-}(\grad f, \xi)}) Qw_{\ell}\|_{H^1(\RnIp)} \gtrsim \|J^{*}(h\partial_y-T_{G^{\e}_{-}(\grad f, \xi)}) Qw_{\ell}\|_{L^2(\RnIp)} \\
&\gtrsim& \|(h\partial_y - T_{G^{\e}_{+}(\grad f, \xi)})(h\partial_y-T_{G^{\e}_{-}(\grad f, \xi)}) Qw_{\ell}\|_{L^2(\RnIp)} \\
& &      -C_{\delta}\|(h\partial_y-T_{G^{\e}_{-}(\grad f, \xi)}) Qw_{\ell}\|_{H^1(\RnIp)} \\
&\gtrsim& \|\Lphet Qw_{\ell}\|_{L^2(\RnIp)} - h\|E_1 Qw_{\ell}\|_{L^2(\RnIp)} -C_{\delta}\|(h\partial_y-T_{G^{\e}_{-}(\grad f, \xi)}) Qw_{\ell}\|_{H^1(\RnIp)}, \\
\end{eqnarray*}
so for small enough $\delta$,
\[
\|(h\partial_y-T_{G^{\e}_{-}(\grad f, \xi)}) Qw_{\ell}\|_{H^1(\RnIp)} \gtrsim \|\Lphet Qw_{\ell}\|_{L^2(\RnIp)} - h\|Qw_{\ell}\|_{H^1(\RnIp)}.
\]
Using the boundedness results for $Q$, we have 
\begin{equation}\label{ClaimS3}
\|(h\partial_y-T_{G^{\e}_{-}(\grad f, \xi)}) Qw_{\ell}\|_{H^1(\RnIp)} \gtrsim \|\Lphet Qw_{\ell}\|_{L^2(\RnIp)} - R.
\end{equation}
(This is the analogous statement to Claim S.3 for the large frequency case:  we've factored $\Lphet$ into two operators, one of which has the proper invertibility property.)  

Therefore
\begin{eqnarray*}
& &       \| J \Lphet Qw_{\ell} \|_{H^{1*}(\RnIp)} \\
&\gtrsim& \|\Lphet Qw_{\ell}\|_{L^2(\RnIp)}-h^{\half}\|(h\partial_y-T_{G^{\e}_{-}(\grad f,\xi)}) Qw_{\ell}\|_{H^{\half}(\Rno)} \\
& &       - C_{\delta}\|JE_1'(h\partial_y-T_{G^{\e}_{-}(\grad f,\xi)})Qw_{\ell}\|_{H^{1*}(\RnIp)} - R. \\
\end{eqnarray*}
Now we can use Lemma \ref{H1star} to replace the second last term by 
\[
-C_{\delta}\|(h\partial_y-T_{G^{\e}_{-}(\grad f,\xi)})Qw_{\ell}\|_{H^1(\RnIp)} - h^{\half}C_{\delta}\|(h\partial_y-T_{G^{\e}_{-}(\grad f,\xi)})Qw_{\ell}\|_{H^{\half}(\Rno)}.
\]
The first part can be absorbed into  $\|\Lphet Qw_{\ell}\|_{L^2(\RnIp)}$ using \eqref{ClaimS3}, so 
\begin{eqnarray*}
\| J \Lphet Qw_{\ell} \|_{H^{1*}(\RnIp)} &\gtrsim& \|\Lphet Qw_{\ell}\|_{L^2(\RnIp)} \\
                                         & &       -h^{\half}\|(h\partial_y-T_{G^{\e}_{-}(\grad f,\xi)})Qw_{\ell}\|_{H^{\half}(\Rno)}-R.\\
\end{eqnarray*}
This finishes the proof of Claim L.2.  Now we need to remove the extraneous boundary term.  

\vspace{3mm}
\noindent \textbf{Claim L.4}  
\begin{eqnarray*}
& &        h^{\half}\|(h\partial_y-T_{G^{\e}_{-}(\grad f, \xi)}) Qw_{\ell}\|_{H^{\half}(\Rno)} \\
&\lesssim& \| \Lphet w_{\ell} \|_{H^{1*}(\RnIp)}  +C_{\delta}\|\Lphet Qw_{\ell}\|_{L^2(\RnIp)} + R. \\
\end{eqnarray*}
\vspace{3mm}

As in the previous claim, we'll instead prove that 
\begin{eqnarray*}
& &        h^{\half}\|(h\partial_y-T_{G^{\e}_{-}(\grad f, \xi)}) Qw_{\ell}\|_{H^{\half}(\Rno)} \\
&\lesssim& \| J \Lphet Qw_{\ell} \|_{H^{1*}(\RnIp)}  +C_{\delta}\|\Lphet Qw_{\ell}\|_{L^2(\RnIp)} + R. \\
\end{eqnarray*}
and use \eqref{JLphetQ_Lphet}.

Returning to the inequality \eqref{L4StartPt} from the proof of Claim L.2 and considering the commutator of $J$ and $(h\partial_y - T_{G^{\e}_{-}(\grad f, \xi)})$, we get
\begin{eqnarray*}
& &       \| J \Lphet Qw_{\ell} \|_{H^{1*}(\RnIp)} \\
&\gtrsim& \|J^{*}(h\partial_y - T_{G^{\e}_{-}(\grad f, \xi)})J Qw_{\ell} \|_{H^{1*}(\RnIp)} - h\|J^{*}E_1'' Qw_{\ell} \|_{H^{1*}(\RnIp)} \\
& &       - C_{\delta}\|JE_1'(h\partial_y-T_{G^{\e}_{-}(\grad f,\xi)})Qw_{\ell}\|_{H^{1*}(\RnIp)} - R. \\
\end{eqnarray*}
$JQ$ is nearly the identity; more precisely $JQ = I + h E_{-1}$ for some order $-1$ operator $E_{-1}$.  Together with Lemma \ref{H1star}, this gives us
\begin{eqnarray*}
& &       \| J \Lphet Qw_{\ell} \|_{H^{1*}(\RnIp)} \\
&\gtrsim& h^{\half} \|(h\partial_y - T_{G^{\e}_{-}(\grad f, \xi)}) w_{\ell} \|_{H^{-\half}(\Rno)} \\
& &       - h\|(h\partial_y - T_{G^{\e}_{-}(\grad f, \xi)}) E_{-1} Pw_{\ell} \|_{L^2(\RnIp)} - h^{\frac{3}{2}}\|E_1'' Qw_{\ell} \|_{H^{-\half}(\Rno)} \\
& &       - C_{\delta}\|E_1'(h\partial_y-T_{G^{\e}_{-}(\grad f,\xi)})Qw_{\ell}\|_{L^2(\RnIp)} -h\|E_1'' Qw_{\ell} \|_{L^2(\RnIp)} \\
& &       - h^{\half}C_{\delta}\|E_1'(h\partial_y-T_{G^{\e}_{-}(\grad f,\xi)})Qw_{\ell}\|_{H^{-\half}(\Rno)} - R. \\
\end{eqnarray*}
Using boundedness results for the various operators involved, we get
\begin{eqnarray*}
& &       \| J \Lphet Qw_{\ell} \|_{H^{1*}(\RnIp)} \\
&\gtrsim& h^{\half} \|(h\partial_y - T_{G^{\e}_{-}(\grad f, \xi)}) w_{\ell} \|_{H^{-\half}(\Rno)} - R \\
& &       - C_{\delta}\|\Lphet Qw_{\ell}\|_{L^2(\RnIp)} - h^{\half}C_{\delta}\|(h\partial_y-T_{G^{\e}_{-}(\grad f,\xi)})Qw_{\ell}\|_{H^{\half}(\Rno)}. \\
\end{eqnarray*}

Now we invoke the boundary conditions on $w$. Since $w$ satisfies the boundary conditions \eqref{tildeBC}, $w_{\ell}$ does as well, with a different $\sigma$.  Therefore on $\Rno$, 
\[
h\partial_y w_{\ell} = \frac{1+ h\sigma + \grad f \cdot h\grad }{1 + |\grad f|^2}w_{\ell}.
\]
Then
\begin{equation}\label{penL4}
\begin{split}
        & \| J \Lphet Qw_{\ell} \|_{H^{1*}(\RnIp)} \\
\gtrsim & h^{\half}\|(1+ \grad f \cdot h\grad -(1+|\grad f|^2)T_{G^{\e}_{-}(\grad f, \xi)}) w_{\ell} \|_{H^{-\half}(\Rno)} - R \\
        & - C_{\delta}\|\Lphet Qw_{\ell}\|_{L^2(\RnIp)} - h^{\half}C_{\delta}\|(h\partial_y-T_{G^{\e}_{-}(\grad f,\xi)})Qw_{\ell}\|_{H^{\half}(\Rno)} \\
\gtrsim & h^{\half}\|T_{\Phi}(1+\grad f \cdot h\grad-(1+|\grad f|^2)T_{G^{\e}_{-}(\grad f, \xi)}) w_{\ell} \|_{H^{\half}(\Rno)} - R \\
        & - C_{\delta}\|\Lphet Qw_{\ell}\|_{L^2(\RnIp)} - h^{\half}C_{\delta}\|(h\partial_y-T_{G^{\e}_{-}(\grad f,\xi)})Qw_{\ell}\|_{H^{\half}(\Rno)} \\
\end{split}
\end{equation}
Now since $w_{\ell}$ satisfies \eqref{tildeBC}, $Qw_{\ell}$ does as well, and so on $\Rno$,
\begin{eqnarray*}
h\partial_y Qw_{\ell} &=& \frac{1+ h\sigma + \grad f \cdot h\grad }{1 + |\grad f|^2}(Qw_{\ell}) \\
                      &=& T_{\Phi} \frac{1+ h\sigma + \grad f \cdot h\grad }{1 + |\grad f|^2}(w_{\ell}) + hE_{-1} w_{\ell}. \\
\end{eqnarray*}
Moreover, on $\Rno$,
\[
T_{G^{\e}_{-}(\grad f,\xi)}Qw_{\ell} = T_{G^{\e}_{-}(\grad f,\xi)}T_{\Phi}w_{\ell} = T_{\Phi}T_{G^{\e}_{-}(\grad f,\xi)}w_{\ell} + hE'_{-1}w_{\ell},
\]
so
\begin{eqnarray*}
& & \|(h\partial_y - T_{G^{\e}_{-}(\grad f, \xi)}) Q w_{\ell} \|_{H^{\half}(\Rno)}\\
&\lesssim& \|T_{\Phi}(1+ \grad f \cdot h\grad -(1+|\grad f|^2)T_{G^{\e}_{-}(\grad f, \xi)}) w_{\ell} \|_{H^{\half}(\Rno)} + R.\\
\end{eqnarray*}
If we substitute this into \eqref{penL4}, then for small enough $\delta$,  
\begin{eqnarray*}
\| J \Lphet Qw_{\ell} \|_{H^{1*}(\RnIp)} &\gtrsim& h^{\half}\|(h\partial_y - T_{G^{\e}_{-}(\grad f, \xi)}) Q w_{\ell} \|_{H^{\half}(\Rno)} \\
                                         & &       - C_{\delta}\|\Lphet Qw_{\ell}\|_{L^2(\RnIp)} - R .\\
\end{eqnarray*}
This completes the proof of Claim L.4.  
\vspace{3mm}

Now we can complete the proof of Lemma \ref{largelemma} by combining the claims and absorbing extraneous terms, as in the small frequency case.

\end{proof}

This completes the proof of Proposition \ref{FlatHMinus1}.  Now by changing variables back to $\Om$, we get the following proposition.

\begin{prop}\label{GraphCarl}
Suppose $w \in H^1(\Om)$ satisfies \eqref{BC}, and $\Gamma_{+}^c$ coincides with a graph of the form $y = f(x)$, where $|\grad f - K| < \d$ for some constants $K \in \Rn$ and $\d > 0$.  
If $\delta$ is small enough, then
\begin{equation*}
h^{\half} \|w\|_{L^2(\Gamma_{+}^c)} + \frac{h}{\sqrt{\e}}\|w\|_{L^2(\Om)} \lesssim \| \Lphe w \|_{H^{1*}(\Om)}.
\end{equation*}
\end{prop}

\subsection{Finishing the Proof of Theorem \ref{LinearCarl}}

Now suppose $\Gamma_{+}$ is as in the hypotheses of Theorem \ref{LinearCarl}, with no extra conditions.  Since $\Gamma_{+}$ is a neighbourhood of $\partial \Om_{+}$, it follows that on $\Gamma_{+}^c$, $\partial_{\nu}\ph < c < 0$ for some $c < 0$.  Therefore locally $\Gamma_{+}^c$ is a graph of the form $y = f(x)$, with $\Om$ lying above the graph.  Moreover, in small enough neighbourhoods, $f$ can be made to obey the graph conditions put on $f$ in the last subsection.  In other words, at any point $p \in \Gamma_{+}^c$, there exists some neighbourhood $U \subset \RnI$ of $p$ such that $\Gamma_{+}^c \cap U$ coincides with a graph of the form $y = f(x)$, with $\Om \cap U$ lying in the set $y > f(x)$, and $|\grad f - K| < \d$, where $K$ is some constant, and $\d$ is small enough for Proposition \ref{GraphCarl} to hold.

Since $\Gamma_{+}^c$ is compact, we can take a finite open cover $U_1, \ldots, U_{m-1}$ of such open sets, and augment it by $U_m$ such that $U_1, \ldots U_m$ is an open cover of $\overline{\Om}$, and $U_m \cap \Gamma_{+}^c$ is empty.  Then 
\[
h^{\half} \|v_j\|_{L^2(\Gamma_{+}^c \cap U_j)} + \frac{h}{\sqrt{\e}}\|v_j\|_{L^2(\Om \cap U_j)} \lesssim \| \Lphe v_j \|_{H^{1*}(\Om)}
\]
holds for all $v_j \in H^1(\Om \cap U_j)$ such that 
\begin{equation}\label{jBC}
\begin{split}
v_j, \partial_{\nu} v_j &= 0 \mbox{ on } \partial (U_j \cap \Om) \setminus \Gamma^c  \\
h\partial_{\nu} (e^{-\frac{\ph}{h}} v_j) &= h\s e^{-\frac{\ph}{h}} v_j \mbox{ on } \Gamma^c \cap U_j.
\end{split}
\end{equation}
Now let $\chi_1, \ldots \chi_m$ be a partition of unity subordinate to $U_1, \ldots U_m$, and for $w \in H^1(\Om)$ satisfying \eqref{BC}, define $w_j = \chi_j w$.  Then $w_j$ satisfies \eqref{jBC} for some $\s$, and so 
\[
h^{\half} \|w_j\|_{L^2(\Gamma_{+}^c \cap U_j)} + \frac{h}{\sqrt{\e}}\|w_j\|_{L^2(\Om)} \lesssim \| \Lphe w_j \|_{H^{1*}(\Om)}.
\]
Adding these estimates together gives
\[
h^{\half} \|w\|_{L^2(\Gamma_{+}^c)} + \frac{h}{\sqrt{\e}}\|w\|_{L^2(\Om)} \lesssim \sum_{j=1}^m \| \Lphe w_j \|_{H^{1*}(\Om)}.
\]
Now 
\begin{eqnarray*}
\| \Lphe w_j \|_{H^{1*}(\Om)} &=&        \| \Lphe \chi_j w \|_{H^{1*}(\Om)} \\
                              &\lesssim& \| \chi_j \Lphe w \|_{H^{1*}(\Om)} + h\| E_1 w \|_{H^{1*}(\Om)}\\
\end{eqnarray*}
where $E_1$ is a first order differential operator.  Then by Lemma \ref{H1star}, 
\begin{eqnarray*}
\| \Lphe w_j \|_{H^{1*}(\Om)} &\lesssim& \| \chi_j \Lphe w \|_{H^{1*}(\Om)} + h\| w \|_{L^2(\Om)} +  h^{\frac{3}{2}}\| w \|_{L^2(\partial \Om)} \\
                              &\lesssim& \| \Lphe w \|_{H^{1*}(\Om)} + h\| w \|_{L^2(\Om)} +  h^{\frac{3}{2}}\| w \|_{L^2(\Gamma_{+}^c)}. \\
\end{eqnarray*}
Therefore
\[
h^{\half} \|w\|_{L^2(\Gamma_{+}^c)} + \frac{h}{\sqrt{\e}}\|w\|_{L^2(\Om)} \lesssim \| \Lphe w \|_{H^{1*}(\Om)} + h\| w \|_{L^2(\Om)} +  h^{\frac{3}{2}}\| w \|_{L^2(\Gamma_{+}^c)},
\]
and the last two terms can be absorbed back into the left side to give
\begin{equation}\label{LpheCarl}
h^{\half} \|w\|_{L^2(\Gamma_{+}^c)} + \frac{h}{\sqrt{\e}}\|w\|_{L^2(\Om)} \lesssim \| \Lphe w \|_{H^{1*}(\Om)}.
\end{equation}
Now we want to replace $\Lphe$ with $\Lphaqe$.  The two operators are related by 
\[
\Lphaqe = \Lphe + 2hA \cdot hD  + 2ihA \cdot \grad \ph_c + h^2(A^2 + q + (D \cdot A)),
\]
so
\[
\| \Lphaqe w \|_{H^{1*}(\Om)} \gtrsim \| \Lphe w \|_{H^{1*}(\Om)} - h\|A \cdot hD w\|_{H^{1*}(\Om)} - h\|w\|_{H^{1*}(\Om)} - h^2\|w\|_{H^{1*}(\Om)}.
\]
The last two terms are bounded by $h\|w\|_{L^2(\Om)}$, so
\[
\| \Lphaqe w \|_{H^{1*}(\Om)} \gtrsim \| \Lphe w \|_{H^{1*}(\Om)} - h\|A \cdot hD w\|_{H^{1*}(\Om)} - h\|w\|_{L^2(\Om)}.
\]
Moreover by Lemma \ref{H1star}, 
\[
h\|A \cdot hD w\|_{H^{1*}(\Om)} \lesssim h\|w\|_{L^2(\Om)} + h^{\frac{3}{2}}\|w\|_{L^2(\partial \Om)}, 
\]
so
\[
\| \Lphaqe w \|_{H^{1*}(\Om)} \gtrsim \| \Lphe w \|_{H^{1*}(\Om)} - h\|w\|_{L^2(\Om)} - h^{\frac{3}{2}}\|w\|_{L^2(\partial \Om)}.
\]
Substituting this into \eqref{LpheCarl} gives
\[
h^{\half} \|w\|_{L^2(\Gamma_{+}^c)} + \frac{h}{\sqrt{\e}}\|w\|_{L^2(\Om)} \lesssim \| \Lphaqe w \|_{H^{1*}(\Om)},
\]
where the missing terms have been absorbed into their counterparts on the left side.  Finally, if $w$ satisfies \eqref{BC} then so does $e^{\frac{\ph^2}{2\e}}w$, so
\[
h^{\half} \|e^{\frac{\ph^2}{2\e}} w\|_{L^2(\Gamma_{+}^c)} + \frac{h}{\sqrt{\e}}\|e^{\frac{\ph^2}{2\e}} w\|_{L^2(\Om)} \lesssim \| e^{\frac{\ph^2}{2\e}} \Lphaq w \|_{H^{1*}(\Om)}.
\]
Then using the boundedness of $e^{\frac{\ph^2}{2\e}}$ on $\Om$, we get
\[
h^{\half} \|w\|_{L^2(\Gamma_{+}^c)} + h\|w\|_{L^2(\Om)} \lesssim \| \Lphaq w \|_{H^{1*}(\Om)}.
\]
This finishes the proof of Theorem \ref{LinearCarl}.

\section{The Logarithmic Case}

Now we turn to the proof of Theorem \ref{MainCarl}.  Following ~\cite{DKSaU} (see in particular Remark 2.8), it suffices, by a change of variables, to work in the following setting.  Let $M_0$ be a smooth compact $n$ dimensional Riemannian manifold with Riemannian metric $g_0$, and let $T = M_0 \times \R$ be equipped with the metric $g = c(g_0 \oplus e)$, where $c > 0$ is a conformal factor.  Let $\Om$ be a smooth domain compactly contained in $T$.  Using the coordinates $(x,y)$ on $T$, where $x \in M_0$ and $y \in \R$, set $\ph(x,y) = y$.  Then we need to prove the following Carleman estimate.  

\begin{theorem}\label{ManifoldCarl}
Define $\partial \Om_{+}$ relative to $\ph$ as before.  Let $\Gamma_{+}$ be a neighbourhood of $\partial \Om_{+}$.  Let $w \in H^1(\Om)$ be such that
\begin{equation}\label{ManBC}
\begin{split}
w, \partial_{\nu} w &= 0 \mbox{ on } \Gamma_{+} \\
h\partial_{\nu} w &= (\partial_{\nu}\ph) w + h\s w \mbox{ on } \Gamma_{+}^c \\
\end{split}
\end{equation}
for some order zero operator $\s$ bounded uniformly in $h$.  There exists $h_0 > 0$ such that if $0<h<h_0$, then
\begin{equation}\label{theMnfldCarl}
h^{\half} \|w\|_{L^2(\Gamma_{+}^c)} + h\|w\|_{L^2(\Om)} \lesssim \| \Lphaq w \|_{H^{1*}(\Om)},
\end{equation}
where $\Lphaq$ is the conjugated operator
\[
\Lphaq = h^2 e^{\frac{\ph}{h}}\Laq e^{-\frac{\ph}{h}},
\]   
and $\Laq$ is as given in \eqref{Laq}, but with $D$ defined in terms of the connection $\grad$ on $T$.  
\end{theorem}

As in the linear case, we will make a series of reductions here.  Firstly, it suffices to prove that for $w \in H^1(\Om)$ satisfying \eqref{ManBC}, 
\[
h^{\half} \|w\|_{L^2(\Gamma_{+}^c)} + \frac{h}{\sqrt{\e}}\|w\|_{L^2(\Om)} \lesssim \| \Lphe w \|_{H^{1*}(\Om)},
\]
where $\Lphe$ is the conjugated operator
\[
\Lphaq = h^2 e^{\frac{\ph_c}{h}}\Lap e^{-\frac{\ph_c}{h}},
\]   
and $\Lap$ is the Laplace-Beltrami operator on $T$.  Then Theorem \ref{ManifoldCarl} follows, since introducing $A$ and $q$ gives rise to errors which can be absorbed into the terms on the left hand side.  

Secondly, we can assume, as in ~\cite{CST}, that the conformal factor $c$ in the metric on $T$ is identically equal to 1.  Finally, as in the proof of Theorem \ref{LinearCarl}, it suffices to divide the domain into pieces, and prove the estimate on each piece separately.  Therefore we may as well assume that there is a choice of coordinates on $\Om$ such that $g_0$ is nearly the Euclidean metric, and $\Gamma_{+}^c$ coincides with a graph of the form $y= f(x)$, where $f$ is smooth.  Then as in ~\cite{CST}, we can change variables twice, first by $(x,y) \mapsto (x, y-f(x))$, and then by the choice of coordinates on $T$.  This maps $\Om$ to a domain $\tilde{\Om}$ in $\R^{n+1}_+$, and $\Gamma_+$ to a subset of $\R^n_0$.  Now it suffices to prove the following proposition.

\begin{prop}\label{LogFlatHMinus1}
Suppose $w \in H^1(\tilde{\Om})$, and 
\begin{equation}\label{logtildeBC}
\begin{split}
w, \partial_{\nu} w &= 0 \mbox{ on } \tilde{\Gamma}_{+} \\
h \partial_{y}w|_{\tilde{\Gamma}_+^c} &= \frac{w + \b \cdot h \grad_{g_0} w - h\s w}{1 + |\g |^2}.
\end{split}
\end{equation}
where $\s$ is smooth and bounded on $\tilde{\Om}$, and $\b$ and $\g$ are a vector valued and scalar valued function, respectively, which coincide with the coordinate representations of $\grad_{g_0} f$ and $|\grad_{g_0} f|_{g_0}$.  
There exists $h_0 > 0$ such that if $0<h<h_0$, then
\begin{equation*}
h^{\half} \|w\|_{L^2(\tilde{\Gamma}_{+}^c)} + \frac{h}{\sqrt{\e}}\|w\|_{L^2(\tilde{\Om})} \lesssim \| \Lphet w \|_{H^{1*}(\RnIp)}.
\end{equation*}
where
\[
\Lphet  = (1+|\g|^2)h^2\partial_y^2 - 2(\a + \b \cdot h \grad_{g_0})h\partial_y + \a^2 + h^2\L,
\]
and $\L$ is the second order differential operator in the $x'$ variables given by 
\[
\L = g_0^{ij}\partial_i \partial_j.
\]
\end{prop}

By our choice of coordinates for the second transformation, we can arrange that for some arbitrary $\delta > 0$, $|g_0 - I| < \delta$ on $\R^{n+1}$, where $I$ is the identity matrix. Since we have divided up the domain into pieces, we can assume also that there is some constant $K$ such that $|\b - K| < \delta$ and $|\g - |K|| < \delta$.

Now our starting point for this proof is the following proposition, which follows from Theorem 1.2 in ~\cite{CST}, applied to the case of zero-forms, after the changes of variables described above.  Here $\tilde{\Om}_2$ is defined in relation to $\tilde{\Om}$ in analogy to the linear case.  

\begin{prop}\label{CVCSTCarl}
Suppose $w \in H^1(\tilde{\Om}_2)$, and $w$ satisfies \eqref{logtildeBC}.  Then there exists $h_0 > 0$ such that if $0<h<h_0$, then
\begin{equation*}
h^{\half} \|w\|_{H^1(\tilde{\Gamma}_{2+}^c)} + \frac{h}{\sqrt{\e}}\|w\|_{H^1(\tilde{\Om}_2)} \lesssim \| \Lphet w \|_{L^2(\RnIp)}.
\end{equation*}
\end{prop}

We are now almost in the same situation as in Section 4, when we had to prove Proposition \ref{FlatHMinus1} followed from Proposition \ref{CofVCarl}, and we will see that most of the proof from Section 4 goes through unchanged.  The main difference is that the second order derivatives in the $x$ variable no longer have constant coefficients.  On the other hand, the coefficients are nearly constant in the sense that $|g_0 - I| < \delta$.

We will define $m_1$, $m_2$, $\mu_1$, $\mu_2$, $\rho$, $w_s$, and $w_{\ell}$ as in Section 4.  Then we need to prove the following lemmas.

\begin{lemma}\label{logsmalllemma}
Suppose $w \in \S(\RnIp)$ has support inside $\tilde{\Om}$ and satisfies the boundary conditions \eqref{logtildeBC}.  Then for appropriate choices of $\delta$, $\mu_1, \mu_2, m_1,$ and $m_2$, 
\begin{equation*}
h^{\half} \|w_{s}\|_{L^2(\Rno)} + \frac{h}{\sqrt{\e}}\|w_{s}\|_{L^2(\RnIp)} \lesssim \| \Lphet w_{s} \|_{H^{1*}(\RnIp)} + h\|w\|_{L^2(\RnIp)} + h\|w\|_{L^2(\Rno)}.
\end{equation*}
\end{lemma} 

\begin{lemma}\label{loglargelemma}
Suppose $w \in \S(\RnIp)$ has support inside $\tilde{\Om}$ and satisfies the boundary conditions \eqref{logtildeBC}.  Then for $\delta$ small enough,
\begin{equation*}
h^{\half} \|w_{\ell}\|_{L^2(\Rno)} + \frac{h}{\sqrt{\e}}\|w_{\ell}\|_{L^2(\RnIp)} \lesssim \| \Lphet w_{\ell} \|_{H^{1*}(\RnIp)}+h\|w\|_{L^2(\RnIp)}+ h\|w\|_{L^2(\Rno)}.
\end{equation*}
\end{lemma} 

As in the linear case, the proof of Proposition \ref{LogFlatHMinus1} will follow from these.  Lemma \ref{logsmalllemma} can be proved in exactly the same manner as Lemma \ref{smalllemma}, since a perturbation of $\Lphet$ by a second order operator with $O(\delta)$ coefficients does not change the proof.  Another way to see this is that in equation \eqref{SubCarl}, we can replace $\Lap_x$ by $\L$ at the cost of adding a $C_{\delta}\|E_2\chi Q w_s\|_{L^2(\RnIp)}$ term to the right hand side.  By the arguments given in the proof, a term of this kind can be absorbed into the left hand side.

However, the proof of Lemma \ref{loglargelemma} requires a change in the definition of $G^{\e}_{\pm}$.

\begin{proof}[Proof of Lemma \ref{loglargelemma}]

We begin by defining $A_{\pm}, \zeta, G_{\pm}, J, \Phi,$ and $Q$ as in the proof of Lemma \ref{largelemma}.  As before, $\chi Q w_{\ell}$ now satisfies \eqref{logtildeBC}, so 
\[
h^{\half} \|\chi Qw_{\ell}\|_{H^1(\tilde{\Gamma}_{2+}^c)} + \frac{h}{\sqrt{\e}}\|\chi Qw_{\ell}\|_{H^1(\tilde{\Om}_2)} \lesssim \| \Lphet \chi Qw_{\ell} \|_{L^2(\tilde{\Om}_2)}.
\]  
and by the arguments for Claim L.1,
\[ 
h^{\half}\| w_{\ell}\|_{L^2(\Rno)}+\frac{h}{\sqrt{\e}}\|w_{\ell}\|_{L^2(\RnIp)} \lesssim \| \Lphet Q w_{\ell} \|_{L^2(\RnIp)}+h\|w\|_{L^2(\RnIp)}+ h^{\frac{3}{2}}\|w\|_{L^2(\Rno)},
\]

Now we will define $G^{\e}_{\pm}$ by 
\[
G^{\e}_{\pm}(\xi) = (1-\zeta)A^{\e}_{\pm}(\xi) + \zeta,
\]
where
\[
A^{\e}_{\pm}(\xi) = \frac{\a + i\b \cdot \xi \pm \sqrt{(\a+i\b \cdot \xi)^2 - (1+|\g|^2)(\a^2-\sum g_0^{ij}\xi_i\xi_j)}}{1 + |\g|^2}.
\]

Then on the support of $w_{\ell}$, $G^{\e}_{\pm}(\xi) = A^{\e}_{\pm}(\xi)$, and $A^{\e}_{\pm}(\xi)$ are the roots of the polynomial
\[
(1 + |\g|^2)X^2 - 2(\a+ i\b \cdot \xi)X + (\a^2 - \sum g_0^{ij}\xi_i\xi_j).
\]
As in the linear case, $G^{e}_{\pm}$ are smooth and symbols of order one. Therefore on the support of $w_{\ell}$ we can factor $\Lphet$ as
\[
(h\partial_y - T_{G^{\e}_{+}(\xi)})(1+|\g|^2)(h\partial_y - T_{G^{\e}_{-}(\xi)})
\]
up to first order error.  Moreover, on the support of $w_{\ell}$, $G^{\e}_{\pm}(\xi) = A^{\e}_{\pm}(\xi)$ is equal to $A_{\pm}(K,\xi)$ up to $O(\d)(1 + |\xi|)$, just like in the linear case, because of the condition that $|g_0 - I| \leq \delta$.  Therefore using these $G^{\e}_{\pm}$, the remainder of the proof of Lemma \ref{largelemma} carries over to the proof of Lemma \ref{loglargelemma}. 

\end{proof}

Thus Proposition \ref{LogFlatHMinus1} and Theorem \ref{ManifoldCarl} follow, and then by a change of variables, we obtain Theorem \ref{MainCarl}.

Note that if $\ph$ is a limiting Carleman weight, then $-\ph$ is a limiting Carleman weight as well.  Replacing $\ph$ with $-\ph$ switches the roles of $\Gamma_{+}$ and $\Gamma_{-}$, so Theorem \ref{MainCarl} yields the following corollary.
\begin{cor}\label{ReverseCarl}
Suppose $w \in H^1(\Om)$, and  
\begin{equation}\label{reverseBC}
\begin{split}
w, \partial_{\nu} w &= 0 \mbox{ on } \Gamma_{-} \\
h\partial_{\nu} w + (\partial_{\nu}\ph)w &= h\s w \mbox{ on } \Gamma_{-}^c. \\
\end{split}
\end{equation}
for some zero order operator $\s$ with uniform bounds in $h$.  Then
\begin{equation*}
h^{\half} \|w\|_{L^2(\Gamma_{-}^c)} + h\|w\|_{L^2(\Om)} \lesssim \| \L_{A,q,-\ph} w \|_{H^{1*}(\Om)}
\end{equation*}
\end{cor}



\section{Complex Geometrical Optics Solutions}

This section will be devoted to the proof of Proposition \ref{CGOs}.  First we need a solvability lemma proved by Hahn-Banach.

\begin{lemma}\label{HBsolns}
For every $v \in L^{2}(\Om)$ and $f \in L^2(\partial \Om)$, there exists $u \in H^1(\Om)$ such that 
\begin{eqnarray*}
\Lphaq  u &=& v  \mbox{ on } \Om\\
(\nu\cdot h(\grad - iA) - \partial_{\nu} \ph)u|_{\Gamma_{-}^c} &=& f 
\end{eqnarray*}
and
\[
\|u\|_{H^1(\Om)} \lesssim h^{-1}\|v\|_{L^2(\RnI)} + h^{\half}\|f\|_{L^2(\partial \Om)}.
\] 
\end{lemma}

\begin{proof}
We follow the methods in, for example, ~\cite{KSU}, but using the Carleman estimate from Corollary \ref{ReverseCarl}.  Let $v \in L^2(\Om)$ and $f \in L^2(\partial \Om)$.  Suppose $w \in H^1(\Om)$ satisfies \eqref{reverseBC}, and consider the expression $(w,v)_{\Om} + (w,hf)_{\partial \Om}.$  We have
\begin{equation*}
\begin{split}
|(w,v)_{\Om} + (w,hf)_{\partial \Om}| &\leq h\|w\|_{L^2(\Om)}h^{-1}\|v\|_{L^2(\Om)} + h^{\half}\|w\|_{L^2(\Gamma_{-}^c)}h^{\half}\|f\|_{L^2(\partial \Om)} \\
                                     &\lesssim \| \L_{\overline{A}, \overline{q},-\ph}w \|_{L^2(\Om)}(h^{-1}\|v\|_{L^2(\Om)}+h^{\half}\|f\|_{L^2(\partial \Om)}), \\
\end{split}
\end{equation*}
with the second inequality being a consequence of Corollary \ref{ReverseCarl}.
Now consider the subspace
\[
\{ \L_{\overline{A}, \overline{q},-\ph}w | w \in H^1(\Om) \mbox{ and } w \mbox{ satisfies \eqref{reverseBC} } \} \subset H^{1*}(\Om).
\]
By Corollary \ref{ReverseCarl}, the linear functional $\L_{\overline{A}, \overline{q},-\ph}w \mapsto (w,v)_{\Om} + (w,hf)_{\partial \Om}$ is well defined on this space.  Then the above estimate shows that it is bounded by $C(h^{-1}\|v\|_{L^2(\Om)}+h^{\half}\|f\|_{L^2(\partial \Om)})$.  Therefore by Hahn-Banach, there is an extension of the functional to $H^{1*}(\Om)$ with the same bound.  Thus there exists $u \in H^1(\Om)$ such that 
\[
\|u\|_{H^1(\Om)} \lesssim h^{-1}\|v\|_{L^2(\Om)}+h^{\half}\|f\|_{L^2(\partial \Om)},
\]
and
\[
(w,v)_{\Om} + (w,hf)_{\partial \Om} = (\L_{\overline{A}, \overline{q},-\ph}w, u).
\]
Integrating by parts on the right side,
\begin{eqnarray*}
(w,v)_{\Om} + (w,hf)_{\partial \Om} &=& (w, \Lphaq u)_{\Om} -h(h\partial_{\nu} w, u)_{\partial \Om} + h(w, h \partial_{\nu} u)_{\partial \Om}\\
                                    & &  - 2h(w, \partial_{\nu} \ph u)_{\partial \Om} - 2h^2(w, i \nu \cdot A u)_{\partial \Om}.\\
\end{eqnarray*}
This holds for all $w \in H^1(\Om)$ which satisfy \eqref{reverseBC}, so in particular it holds for all $w \in C^{\infty}_0(\Om)$.  This means that
\[
\Lphaq u = v
\]
on $\Om$.  Thus
\[
(w,hf)_{\partial \Om} = -h(h\partial_{\nu} w, u)_{\partial \Om} + h(w, (h \partial_{\nu} - 2\partial_{\nu} \ph - 2ih\nu\cdot A) u)_{\partial \Om}.
\]
Using the boundary conditions \eqref{reverseBC}, with $\s = i\nu\cdot A$, we get
\[
(w,hf)_{\Gamma_{-}^c} = h(w, (h \partial_{\nu} - ih\nu \cdot A - \partial_{\nu} \ph) u)_{\Gamma_{-}^c}.
\]
for all $w \in H^2(\Om)$ which satisfy \eqref{reverseBC}.  Therefore 
\[
(h \partial_{\nu}- ih\nu \cdot A - \partial_{\nu} \ph) u|_{\Gamma_{-}^c} = f.
\]
\end{proof}

Now we can construct the CGO solutions from Proposition \ref{CGOs}.  

\begin{proof}[Proof of Proposition \ref{CGOs}]

If $\psi(x,y)$ solves the eikonal equations
\[
\grad \ph \cdot \grad \psi = 0,  |\grad \ph| = |\grad \psi|,
\]
and $a$ is a solution to the Cauchy-Riemann equation
\[
(-\grad \ph + i\grad \psi) \cdot (\grad + iA) a + (\grad + iA) \cdot (-\grad \ph + i\grad \psi) a  = 0,
\]
as in ~\cite{DKSjU}, then
\[
h^2(\Laq)e^{\frac{1}{h}(-\ph + i\psi)}a = O(h^2)e^{\frac{-\ph}{h}}.
\]
Therefore
\[
e^{\frac{\ph}{h}}h^2(\Laq)e^{\frac{1}{h}(-\ph + i\psi)}a = v,
\]
for some $v$ with $\|v\|_{L^2(\Om)} = O(h^2)$.  Moreover, at $\partial \Om$,
\[
e^{\frac{\ph}{h}} \nu \cdot h(\grad + iA) e^{\frac{1}{h}(-\ph + i\psi)}a = g,
\]
where $\|g\|_{L^2(\partial \Om)} = O(1)$.  
Now by Lemma \ref{HBsolns}, there exists a solution $r_0 \in H^1(\Om)$ to the problem
\begin{eqnarray*}
                                                               \Lphaq r_0 &=& -v \\
(h \partial_{\nu} - ih\nu\cdot A - \partial_{\nu} \ph)r_0|_{\Gamma_{-}^c} &=& -g, \\
\end{eqnarray*}
and $\|r_0\|_{H^1(\Om)} = O(h^{\half})$.  Then if $r = e^{-\frac{i\psi}{h}}r_0$, then $\|r\|_{H^1(\Om)} = O(h^{\half})$ and
\begin{eqnarray*}
\Laq e^{\frac{1}{h}(-\ph + i\psi)}(a + r) &=& 0 \\
\nu \cdot h(\grad + iA) e^{\frac{1}{h}(-\ph + i\psi)}(a + r)|_{\Gamma_{-}^c} &=& 0. 
\end{eqnarray*}

This completes the proof.  

\end{proof}

\bibliographystyle{alpha}

\end{document}